\documentclass{IJFV}

\usepackage{amssymb}
\usepackage{graphicx}
\usepackage{stmaryrd}
\usepackage{dsfont}
\usepackage{subcaption}

\usepackage[dvipsnames]{xcolor}

\newcommand{\R}{\ensuremath{\; \mathbb{R}}}

\newcommand{\N}{ \ensuremath{\mathbb{N}}}

\newcommand{\dx}{\mathrm{d}x}

\newcommand{\dt}{\mathrm{d}t}

\newcommand{\dpsi}[2]{\mathrm{d}\Psi_{#1 \frac{1}{2}}^{#2}}

\newcommand{\flux}[2]{\mathcal{F}_{u,\,#1 \frac{1}{2}}^{#2}}

\newcommand{\h}[1]{h_{#1 \frac{1}{2}}}

\newcommand{\sn}[2]{\sum\limits_{k=#1}^{#2 - 1}}

\newcommand{\si}[2]{\sum\limits_{i=#1}^{#2}}

\newcommand{\som}[3]{\sum\limits_{#3=#1}^{#2}}

\newcommand{\intt}[3]{\int\limits_{#1}^{#2} #3 \mathrm{d}t}

\newcommand{\intm}[4]{\int\limits_{#1}^{#2} #3 \, \mathrm{d}#4}

\newcommand{\Htau}{\mathcal{H}_{\mathcal{T}}}

\newcommand{\Htaudelta}{\mathcal{H}_{\mathcal{T},\, \Delta t}}

\newcommand{\Htaum}{\mathcal{H}_{\mathcal{T}_m}}

\newcommand{\Htaudeltam}{\mathcal{H}_{\mathcal{T}_m,\, \Delta t_m}}

\newcommand{\tild}[1]{\widetilde{#1}}

\newcommand{\fhm}[1]{#1_{h_m}}

\newcommand{\bhat}[1]{\widehat{#1}}

\newcommand{\Dpsi}[1]{\Delta \Psi_{#1}^{pzc}}

\newcommand{\normeldeuxhun}[1]{\Vert #1 \Vert_{0; 1, \mathcal{T}}}

\newcommand{\normeun}[1]{\Vert #1 \Vert_{1, \, \mathcal{T}}}

\newcommand{\normeunm}[1]{\Vert #1 \Vert_{1, \, \mathcal{T}_m}}

\newcommand{\normeduale}[1]{\Vert #1 \Vert_{-1,2, \, \mathcal{T}}}

\newcommand{\normedualem}[1]{\Vert #1 \Vert_{-1,2, \, \mathcal{T}_m}}

\newcommand{\normezero}[1]{\Vert #1 \Vert_{0}}

\newcommand{\normelinftyzeroun}[1]{\Vert #1 \Vert_{\mathrm{L}^{\infty}(0,1)}}

\theoremstyle{plain}
\newtheorem{theo}{Theorem}[section]

\theoremstyle{plain}
\newtheorem{prop}{Proposition}[section]

\theoremstyle{plain}
\newtheorem{lem}{Lemma}[section]

\theoremstyle{plain}
\newtheorem{cor}{Corollary}[section]

\theoremstyle{definition}
\newtheorem{defi}{Definition}[section]

  {\medbreak\par}
  {\medbreak\par}

\begin{document}

\IJFVtitle{Finite volume scheme for a corrosion model}{Convergence of a finite volume scheme for a corrosion model}

\IJFVauthor{Claire Chainais-Hillairet} 
\IJFVinstitution{Laboratoire Paul Painlev\'e, CNRS UMR 8524, Universit\'e Lille 1,\newline 59655 Villeneuve d'Ascq Cedex, France} 
\IJFVemail{claire.chainais@math.univ-lille1.fr}
\IJFVauthor{Pierre-Louis Colin}
\IJFVinstitution{Laboratoire Paul Painlev\'e, CNRS UMR 8524, Universit\'e Lille 1,\newline 59655 Villeneuve d'Ascq Cedex, France}
\IJFVemail{pierre-louis.colin@ed.univ-lille1.fr}
\IJFVauthor{Ingrid Lacroix-Violet}
\IJFVinstitution{Laboratoire Paul Painlev\'e, CNRS UMR 8524, Universit\'e Lille 1,\newline 59655 Villeneuve d'Ascq Cedex, France }
\IJFVemail{ingrid.violet@univ-lille1.fr}
\IJFVauthor{Pierre-Louis Colin$^\dagger$}
\IJFVinstitution{$^\dagger$Laboratoire Paul Painlev\'e, CNRS UMR 8524, Universit\'e Lille 1,\newline 59655 Villeneuve d'Ascq Cedex, France}
\IJFVemail{pierre-louis.colin@ed.univ-lille1.fr}
\IJFVauthor{Ingrid Lacroix-Violet$^\dagger$}
\IJFVinstitution{$^\dagger$Laboratoire Paul Painlev\'e, CNRS UMR 8524, Universit\'e Lille 1,\newline 59655 Villeneuve d'Ascq Cedex, France}
\IJFVemail{ingrid.violet@univ-lille1.fr}

\IJFVabstract{
In this paper, we study the numerical approximation of a system of partial differential equations describing the corrosion of an iron based alloy in a nuclear waste repository. In particular, we are interested in the convergence of a numerical scheme consisting in an implicit Euler scheme in time and a Scharfetter-Gummel finite volume scheme in space.
}

\IJFVkeywords{finite volume scheme, corrosion model, convergence analysis, drift-diffusion system}


\section{Introduction}


\subsection{General framework of the study}

At the request of the French nuclear waste management agency ANDRA, investigations are conducted to evaluate the long-term safety assessment of the geological repository of high-level radioactive waste. The concept of the storage under study in France is the following: the waste is confined in a glass matrix, placed into cylindrical steel canisters and stored in a claystone layer at a depth of several hundred of meters. The long-term safety assessment of the geological repository has to take into account the degradation of the carbon steel used for waste overpacks, which is mainly caused by generalized corrosion processes.

In this framework, the Diffusion Poisson Coupled Model (DPCM) has been proposed by C. Bataillon {\it et  al.} in \cite{BCD10} in order to describe corrosion processes at the surface of the steel canisters. It assumes that the metal is covered by a dense oxide layer which is in contact with the claystone. The model describes the evolution of the dense oxide layer. In most industrial cases, the shape of metal pieces is not flat (container or pipe). But the oxide layer is very thin compared to the size of the exposed surface. In practice,
the available data are averaged over the whole exposed surface and the microscopic or
even macroscopic heterogeneities of materials are not taken into account. For these reasons, a 1D modeling has been proposed in order to describe the real system.

The oxide layer behaves as a semiconductor: charge carriers, like electrons, cations $F\!e^{3+}$ and oxygen vacancies,  are convected by the electric field and the electric potential is coupled to the charge densities. The DPCM model is then made of drift-diffusion equations on the charge densities coupled with a Poisson equation on the electric potential. The boundary conditions induced by the electrochemical reactions at the interfaces are Robin boundary conditions. Moreover, the system includes moving boundary equations.

Numerical methods for the approximation of the DPCM model  have been designed and studied  by Bataillon {\em et al } in \cite{BBC12}. Numerical experiments with real-life data shows the ability of the model to reproduce expected physical behaviors. However, the proof of convergence of the scheme is challenging. In this paper, we will focus on a simplified model with only two species: electrons and cations $F\!e^{3+}$. As the displacement of the interfaces in the full DPCM model is due to the current of oxygen vacancies, which are not taken into account in this study, the simplified model will be posed on a fixed domain. These simplifications will permit us to show how to deal with the boundary conditions for the numerical analysis of the scheme introduced in \cite{BBC12}.


\subsection{Presentation of the model}

In this paper we focus on the simplified corrosion model already introduced  in \cite{BBC12, CV13}.  The unknowns of the model are the densities of electrons $N$ and cations $F\!e^{3+}$ $P$ and the electric potential $\Psi$. The current densities are respectively denoted by $J_{N}$ and $J_{P}$; they contain both a drift and a diffusion part. Therefore, the model consists in two drift-diffusion equations on the densities, coupled with a Poisson equation on the electric potential. Let $T>0$, the model is written as 
\begin{subequations} \label{modele}
\begin{gather}
\partial_{t} P + \partial_{x} J_{P} = 0 , \qquad J_{P} = -\partial_{x} P - 3 P \partial_{x} \Psi, \quad  \text{in} \ (0,1)\times(0,T), \label{equP} \\
\varepsilon \partial_{t} N + \partial_{x} J_{N} = 0 , \qquad J_{N} = -\partial_{x} N + N \partial_{x} \Psi, \quad \text{in} \  (0,1)\times(0,T), \label{equN} \\
- \lambda^{2} \partial_{xx}^{2} \Psi = 3 P - N + \rho_{hl}, \quad \text{in} \  (0,1)\times(0,T), \label{equpsi}
\end{gather}
\end{subequations}
where $ \lambda$ is the rescaled Debye length and $ \rho_{hl} $ the net charge density of the ionic species in the host lattice which is constant. The parameter $ \varepsilon $ stands for the ratio of the mobility coefficients of electrons and cations, then $ \varepsilon \ll 1 $. 

As the equations (\ref{equP}) and (\ref{equN}) for charge carriers densities have the same form, we will use the following synthetical form:
\begin{align}
\varepsilon_u \partial_{t} u + \partial_{x} J_{u} = 0 , \qquad J_{u} = -\partial_{x} u - z_u u \partial_{x} \Psi, \quad \text{in} \ (0,1)\times(0,T). \label{equu}
\end{align}
For $ u=P,N $, the charge numbers of the carriers are respectively $z_{u}=3, -1$ and we respectively have $\varepsilon_{u}=1, \varepsilon$.

Let us now focus on the boundary conditions. Charge carriers are created and consumed at both interfaces $x=0$ and $x=1$.  
The kinetics of the electrochemical reactions at the interfaces are described by Butler-Volmer laws. It leads to Robin boundary conditions on $N$ and $P$. As in \cite{BBC12}, we assume that the boundary conditions for $P$ and $N$ have exactly the same form. Therefore, for $u=P,N$, they are written as
\begin{subequations}\label{CLu}
\begin{align}
-J_u=r_u^0(u,\Psi)\quad &\mbox{ on } \left\{x=0\right\}\times (0,T), \label{CLu0}\\
J_u=r_u^1(u,\Psi,V)&\mbox{ on } \left\{x=1\right\}\times (0,T), \label{CLu1}
\end{align}
\end{subequations}
where $V$ is a given applied potential (we just consider here the potentiostatic case) and $r_u^0$ and $r_u^1$ are linear and monotonically increasing functions with respect to their first argument. More precisely, due to the electrochemical reactions at the interfaces, we have, for $u=P,N$: 
\begin{subequations}\label{defru}
\begin{align}
r_u^0(s,x)&=\beta_u^0(x)s-\gamma_u^0(x),\label{defru0}\\
r_u^1(s,x,V)&=\beta_u^1(V-x)s-\gamma_u^1(V-x),\label{defru1}
\end{align}
\end{subequations}
where the functions $(\beta_u^i)_{i=0,1}, (\gamma_u^i)_{i=0,1}$ are given functions. These functions depend on many parameters: the interface kinetic coefficients $(m_u^i, k_u^i)_{i=0,1}$, the positive transfer coefficients $(a_u^i, b_u^i)_{i=0,1}$, the maximum occupancy for octahedral cations in the lattice $P^{max}$ and the electron density in the state of metal $N^{max}$. For $u=P,N$, the functions  $(\beta_u^i)_{i=0,1}, (\gamma_u^i)_{i=0,1}$ are written 
\begin{subequations}\label{defbetagamma}
\begin{align}
\beta_u^i(x)&=m_u^i e^{-z_u b_u^i x}+k_u^ie^{z_u a_u^i x},~i=0,1,\label{betau}\\
\gamma_u^0(x)&=m_u^0u^{max}e^{-z_u b_u^0 x},~\gamma_u^1(x)=k_u^1u^{max}e^{z_u a_u^1 x}.\label{gammau}
\end{align}
\end{subequations}

Throughout the paper, we will assume that the interface kinetic coefficients and the transfer coefficients  are given constants which satisfy
\begin{equation}\label{hyp_mk}
m_u^0, k_u^0, m_u^1, k_u^1 >0, \mbox{ for } u=P,N,
\end{equation}
\begin{equation}\label{hyp_ab}
a_u^0, b_u^0, a_u^1, b_u^1 \in [0,1], \mbox{ for } u=P,N.\end{equation}
We also assume that $\rho_{hl}$ does not depend on $x$ and that
\begin{equation}\label{hyp_NPrho}
3P^{max}-N^{max}+\rho_{hl}=0.
\end{equation}
Indeed, in the applications (see \cite{BCD10}), the scaling of the model leads to $\rho_{hl}=-5$, $P^{max}=2$ and $N^{max}=1$, so that the relation \eqref{hyp_NPrho} is satisfied. 

The boundary conditions for the Poisson equation take into account that the metal and the solution can be charged because they are respectively electronic and ionic conductors. Such an accumulation of charges induces a field given by the Gauss law. These accumulations of charges depend on the voltage drop at the interface given by the usual Helmholtz law which links the charge to the voltage drop through a capacitance. The parameters $\Delta \Psi_0^{pzc}$ and $\Delta \Psi_1^{pzc}$ are the voltage drop corresponding to no accumulation of charges respectively in the metal and in the solution. Finally, the boundary conditions for the electric potential are written:
\begin{subequations}\label{CLpsi}
\begin{align}
\Psi-\alpha_0\partial_x \Psi = \Delta \Psi_0^{pzc},\quad &\mbox{ on } \left\{x=0\right\}\times (0,T), \label{CLpsi0}\\
\Psi+\alpha_1\partial_x \Psi = V-\Delta \Psi_1^{pzc},&\mbox{ on } \left\{x=1\right\}\times (0,T), \label{CLpsi1}
\end{align}
\end{subequations}
where $\alpha_0$ and $\alpha_1$ are positive dimensionless parameters arising in the scaling. 

The system is supplemented with initial conditions, given in $ L^{\infty}(0,1) $:
\begin{align} \label{CI}
u(x,0)=u^{0}(x), \quad \text{for} \ u=P, N.
\end{align}
Moreover, we assume that these initial conditions satisfy
\begin{align}\label{estiLinfu01}
0\leqslant u^0 \leqslant u^{max},  \ \text{a.e. on} \  (0,1), \quad\text{for} \ u=P, N.
\end{align}
In the remainder of the paper we will denote by $(\mathcal{P})$ the corrosion model defined by \eqref{modele}, \eqref{defru}, \eqref{defbetagamma}, \eqref{CLpsi} and \eqref{CI}.
Let us first note that the system of equations \eqref{modele} is  the so-called linear drift-diffusion system. This model is currently used in the framework of semiconductors device modeling (see for instance \cite{vR,Jun01,Mar90,Mar86}) or plasma physics (see \cite{Ch1984}). In this context the drift-diffusion model has been widely studied, from the analytical as from the numerical point of view. Let us refer to the pioneering work by Gajewski \cite{Ga1985} about existence and uniqueness results. Further developments have been done in \cite{Ga1994,FaIt95,Ju1994,DaV96}. The long-time behavior of solutions to the drift-diffusion system {\em via} an entropy method has been studied in \cite{GG96, Jun95} and the stability at the quasi-neutral limit or at the zero-electron-mass limit in \cite{Ga01,JP00,JuPe00_2,JuVi2007}.
 Different methods have also been proposed for the approximation of the drift-diffusion system, and studied, {see for instance \cite{ChLiPe03,CP03} for finite volume schemes, \cite{SaSa97} for a mixed finite volume scheme and \cite{BrMaPi89,BMPcras,BMPsinum,ChCo95} for finite element and mixed exponential fitting schemes.}

In the modeling of semiconductor devices, the boundary conditions are generally mixed Dirichlet/Neumann boundary conditions (corresponding to the ohmic contacts and the insulated boundary segments of the device). Then, the originality of the corrosion model described in this paper lies in the boundary conditions \eqref{defru}, \eqref{CLpsi} which are of Robin type, and induce an additional coupling between the equations.  Let us define the notion of weak solution to the corrosion model $(\mathcal{P})$.
\begin{defi}\label{solfaible}
We say that $ (P, N, \Psi) \in L^2(0,T;{H}^1(0,1)) \cap L^{\infty}([0,T]\times[0,1]) $ is a weak solution of $(\mathcal{P})$, if for all $ \varphi $ in $ L^2(0,T;H^1(0,1))$,
\begin{multline}\label{DDufaible}
-\varepsilon_u\int_0^T \!\!\!\int_0^1 u \partial_t{\varphi} dxdt -\varepsilon_u \int_0^1 u_0(x){\varphi}(0,x)dx-\int_0^T\!\!\! \int_0^1 \left(-\partial_x u-z_uu\partial_x\Psi\right)\partial_x {\varphi} dx dt  \\
+ \int_0^T \left[\left(\beta_u^1(V-\Psi(t,1))u(t,1)-\gamma_u^1(V-\Psi(t,1))\right){\varphi}(t,1) \right. \\
+\left. \left(\beta_u^0(\Psi(t,0))u(t,0)-\gamma_u^0(\Psi(t,0))\right){\varphi}(t,0)\right]dt=0, \mbox{ for } u=P,N,
\end{multline}
and
\begin{multline}\label{poissonfaible}
\lambda^2 \int_0^T \!\!\!\int_0^1 \partial_x \Psi \partial_x {\varphi} dx dt - \int_0^T \frac{\lambda^2}{\alpha_1} \left(V-\Psi(t,1)-\Delta \Psi_1^{pzc}\right){\varphi}(t,1)dt \\
+ \int_0^T \frac{\lambda^2}{\alpha_0} \left(\Psi(t,0)-\Delta \Psi_0^{pzc}\right){\varphi}(t,0)dt =\int_0^T \!\!\!\int_0^1 (3P-N+\rho_{hl}){\varphi} dx dt. 
\end{multline}

\end{defi}

In \cite{CV13}, Chainais-Hillairet and Lacroix-Violet have proved, under some assumptions on the chemical and physical parameters, the existence of a  weak solution  to $(\mathcal{P})$. This result is obtained by passing to the limit in an approximate solution given by a semi-discretization in time. Convergence of the sequence of approximate solutions is ensured by some estimates which yield compactness. In the current article, we apply the same ideas as in \cite{CV13} to a full discretization of $(\mathcal{P})$ presented below.


\subsection{Presentation of the numerical scheme}

Some numerical schemes for the approximation of $(\mathcal{P})$ have been proposed by Bataillon {\em et al}  in \cite{BBC12}. Their stability analysis is fulfilled but no convergence results are given. Here we are interested in the convergence analysis of the fully implicit scheme introduced in \cite{BBC12}. It is a backward Euler scheme in time and a finite volume scheme in space, with Scharfetter-Gummel approximation of the convection-diffusion fluxes. 

Let us consider a mesh $ \mathcal{T} $ for the domain $ [0,1] $. It consists in a family of mesh cells denoted by $ \left( x_{i-\frac{1}{2}} , x_{i+\frac{1}{2}} \right) $ for $i\in \llbracket 1;I\rrbracket$, with  
\begin{align*}
 0 = x_{1/2} < x_{3/2} < \cdots < x_{I-1/2} < x_{I+1/2} = 1.
\end{align*}
Then, we define $ x_{i} = \dfrac{x_{i+1/2} + x_{i-1/2}}{2} $, for $ i\in \llbracket 1;I\rrbracket$ and $x_0=x_{1/2}=0$, $x_{I+1}=x_{I+1/2}=1$. Moreover, we set 
\begin{align*}
& h_{i} = x_{i+\frac{1}{2}} - x_{i-\frac{1}{2}}, \qquad \forall \, i\in \llbracket 1;I\rrbracket,\\
& h_{i+\frac{1}{2}} = x_{i+1} - x_{i}, \qquad \forall \, i\in \llbracket 0;I\rrbracket.
\end{align*}
The mesh size is defined by  $ h= \max \left\{ h_{i}, i\in \llbracket 1;I\rrbracket \right\} $. 

Let us denote by  $ \Delta t $ the time step. We will assume that there exists $K\in\N$ such that $ K \Delta t = T $  (either, we would define $K$ as the integer part of $T/\Delta t$). We  consider the sequence $ (t^k)_{0 \leqslant k \leqslant K} $ such that $ t^k = k \Delta t $. 

The scheme under study in this paper is written as follows. For $i\in \llbracket 1;I\rrbracket$, $k\in \llbracket 0;K-1\rrbracket$, 
\begin{subequations}\label{scheme}
\begin{gather}
- \lambda^{2} \left( \mathrm{d} \Psi_{i + \frac{1}{2}}^{k+1} -\mathrm{d} \Psi_{i - \frac{1}{2}}^{k+1}  \right) = h_{i} \left( 3 P_{i}^{k+1} - N_{i}^{k+1} + \rho_{hl} \right),\label{eqpsinum} \\
\varepsilon_u h_{i} \dfrac{u_{i}^{k+1} - u_{i}^{k}}{\Delta t} + \mathcal{F}_{u,i + \frac{1}{2}}^{k+1} - \mathcal{F}_{u,i - \frac{1}{2}}^{k+1}= 0, \mbox{ for } u=P,N,\label{equnum}
\end{gather}
\end{subequations}
with the numerical fluxes defined for $i\in \llbracket 0;I\rrbracket$ by
\begin{subequations}
\begin{gather}
\mathrm{d} \Psi_{i + \frac{1}{2}}^{k+1} = \dfrac{\Psi_{i+1}^{k+1} - \Psi_{i}^{k+1} }{h_{i+\frac{1}{2}}}, \label{fluxpsinum}\\
\mathcal{F}_{u,i + \frac{1}{2}}^{k+1} = \dfrac{B \left( z_u h_{i+\frac{1}{2}} \mathrm{d} \Psi_{i + \frac{1}{2}}^{k+1} \right) u_{i}^{k+1} - B \left( - z_u h_{i+\frac{1}{2}} \mathrm{d} \Psi_{i + \frac{1}{2}}^{k+1} \right) u_{i+1}^{k+1}}{h_{i + \frac{1}{2}}}, \qquad \mbox{ for } u=P,N, \label{fluxunum}
\end{gather}
\end{subequations}
where $ B $ is the Bernoulli function :
\begin{align*}
B(x) = \dfrac{x}{e^{x} - 1}, \ \forall x \neq 0 \ \ \text{and} \ \ B(0) = 1.
\end{align*}
We supplement the scheme with the discretization of the boundary conditions: for $k\in \llbracket 0;K-1\rrbracket$, 
\begin{subequations}\label{CBdisc}
\begin{gather}
\Psi_{0}^{k+1}  - \alpha_{0} \mathrm{d} \Psi_{\frac{1}{2}}^{k+1}  = \Delta \Psi_{0}^{pzc}, \label{CBPsi0} \\
\Psi_{I+1}^{k+1}  + \alpha_{1} \mathrm{d} \Psi_{I + \frac{1}{2}}^{k+1}  = V - \Delta \Psi_{1}^{pzc}, \label{CBPsi1} \\
- \mathcal{F}_{u,\frac{1}{2}}^{k+1} = \beta_{u}^{0} \left( \Psi_{0}^{k+1} \right) u_{0}^{k+1} - \gamma_{u}^{0} \left( \Psi_{0}^{k+1} \right)  \qquad \mbox{ for } u=P,N, \label{CBFluxu0} \\
 \mathcal{F}_{u,I +\frac{1}{2}}^{k+1} = \beta_{u}^{1} \left( V - \Psi_{I+1}^{k+1}\right) u_{I+1}^{k+1} - \gamma_{u}^{1} \left( V - \Psi_{I+1}^{k+1} \right)  \qquad \mbox{ for } u=P,N, \label{CBFluxu1}
\end{gather}
\end{subequations}
and of the initial conditions: for $i\in \llbracket 1;I\rrbracket$,
\begin{align} \label{CIdisc}
u_{i}^{0} = \dfrac{1}{h_{i}} \int_{x_{i-\frac{1}{2}}}^{x_{i+\frac{1}{2}}} u^{0}(x) \, \mathrm{d}x, \quad  \mbox{ for } u=P,N.
\end{align}
The scheme \eqref{scheme}-\eqref{CIdisc} will be denoted in what follows by $(\mathcal{S})$.

\begin{remark}
The choice of the Bernoulli function for $B$ corresponds to a Scharfetter-Gummel approximation of the convection-diffusion fluxes. These numerical fluxes have been introduced by Il'in in \cite{ilin} and Scharfetter and Gummel in \cite{SG69} for the numerical approximation of the drift-diffusion system arising in semiconductor modelling.   Lazarov, Mishev and Vassilevsky in \cite{LMV96} have established that they are second-order accurate in space.  Dissipativity of the Scharfetter-Gummel scheme with a backward Euler time discretization for the classical drift-diffusion system was proved in \cite{GG96} and Chatard in \cite{Cha11}. One crucial property of the Scharfetter-Gummel fluxes is that they generally preserve steady-states.

\end{remark}

\begin{remark}
The scheme $(\mathcal{S})$ is written on a nonuniform discretization of the domain $[0,1]$. Indeed, the numerical experiments done in \cite{BCH08} for the steady-state of $({\mathcal S})$ show some boundary layers for the density profiles. Therefore, it seems relevant to use  a mesh which is  refined near the boundaries. We will use a Tchebychev mesh, already introduced in \cite{BCH08} and \cite{BBC12}. 

For the discretization in time, it is easier to deal with a fixed time step when studying the convergence. However, it would be also possible to introduce a variable time step. Adaptive time step strategy could also be taken into account, see \cite{BBC12}.
\end{remark}


\subsection{Main results} \label{sectionmainresult}

The aim of this paper is to prove the convergence of a sequence of approximate solutions obtained with the numerical scheme $(\mathcal{S})$ to a solution of $(\mathcal{P})$. 

To this end, we first need to establish the existence of a solution to the scheme. Indeed, at  each time step $k\in \llbracket 0;K-1\rrbracket$, the vector of discrete unknowns \break$ ({\mathbf P}^{k+1}, {\mathbf N}^{k+1}, {\mathbf \Psi}^{k+1})$, with  ${\mathbf P}^{k+1}=(P_i^{k+1})_{0 \leqslant i \leqslant I+1}$,  ${\mathbf N}^{k+1}=(N_i^{k+1})_{0 \leqslant i \leqslant I+1}$ , \break${\mathbf \Psi}^{k+1}=(\Psi_i^{k+1})_{0 \leqslant i \leqslant I+1} $, is defined as a solution to the nonlinear system of equations \eqref{scheme}--\eqref{CBdisc}. In \cite{BBC12}, the existence of a solution has been proved  only in the case where $\varepsilon >0$. 
As stated in Proposition \ref{PROP}, the result holds even if $\varepsilon=0$. 

\begin{prop}\label{PROP}
Let $\varepsilon \geq 0$, $ \alpha_0 > 0 $, $ \alpha_1 > 0 $ and the hypotheses \eqref{hyp_mk}, \eqref{hyp_ab}, \eqref{hyp_NPrho}, \eqref{estiLinfu01} hold. 
Let us assume that
\begin{align}
& - \dfrac{1}{3a_P^0} \left( 1 + \log \left( \alpha_0 a_P^0 k_P^0 \right) \right) \leqslant \Delta \Psi_0^{pzc} \leqslant \dfrac{1}{a_N^0} \left( 1 + \log \left( \alpha_0 a_N^0 k_N^0 \right) \right), \label{Dpsi0} \\
& - \dfrac{1}{b_N^1} \left( 1 + \log \left( \alpha_1 b_N^1 m_N^1 \right) \right) \leqslant \Delta \Psi_1^{pzc} \leqslant \dfrac{1}{3 b_P^1} \left( 1 + \log \left( \alpha_1 b_P^1 m_P^1 \right) \right). \label{Dpsi1}
\end{align}
Then there exists a solution $ ({\mathbf P}^{k+1}, {\mathbf N}^{k+1}, {\mathbf \Psi}^{k+1})_{0\leq k \leq K-1} $ to the fully implicit scheme $(\mathcal{S})$. Moreover, it satisfies the following stability property:
\begin{align}\label{estiLinfu}
0\leqslant P_i^{k} \leqslant P^{max} \ \ \text{and} \ \ 0 \leqslant N_i^{k} \leqslant N^{max}, \qquad \forall i\in \llbracket 1;I+1\rrbracket, \; \forall k \in \llbracket 0;K\rrbracket.
\end{align}
\end{prop}

Based on the vector of discrete unknowns, we can define some approximate solutions, which are piecewise constant function in space and time, as it is usual for finite volume approximations. 
For a given mesh $\mathcal{T}$ of size $h$ and a given $\Delta t$,  we define,  for $w=N, P$ or $\Psi$,
\begin{align}
& w_h^k=\sum_{i=1}^I w_i^{k} \mathds{1}_{\left(x_{i-1/2},x_{i+1/2}\right)}+w_0^k\mathds{1}_{\{x=0\}}+w_{I+1}^k\mathds{1}_{\{x=1\}}, ~\hbox{for }~k \in\llbracket 0;K\rrbracket , \label{wrecspace} \\
& w_{h,\Delta t}=\sum_{k=0}^{K-1} w_{h}^{k+1} \mathds{1}_{[t^k,t^{k+1})}. \label{wrecspacetime}
\end{align}

%

For a sequence of meshes and time steps $ (\mathcal{T}_m, \Delta t_m)_m $ such that $h_m \to 0 $ and $ \Delta t_m \to 0 $ as $ m \to  + \infty $, we can define a sequence of approximate solutions \break$(P_m,N_m,\Psi_m)_m $ with $w_m=w_{h_m,\Delta t_m}$ for $w=P, N$ or $\Psi$. The main result of the paper is the convergence of such a sequence of approximate solutions to a weak solution of the corrosion model $({\mathcal P})$. It is given in Theorem \ref{THEO}.
\begin{theo}\label{THEO}
Let $\varepsilon > 0$, $ \alpha_0 > 0 $, $ \alpha_1 > 0 $. Assuming (\ref{hyp_mk}), (\ref{hyp_ab}), (\ref{hyp_NPrho}), (\ref{estiLinfu01}), \eqref{Dpsi0} and \eqref{Dpsi1}, there exist $ P, N $ and $ \Psi  \in L^2(0,T;H^1(0,1)) $ such that, up to a subsequence, as $ m \to + \infty $, 
\begin{align*}
& P_m \to P \quad \text{strongly in} \ L^2(0,T;L^2(0,1)), \\
& N_m \to N \quad \text{strongly in} \ L^2(0,T;L^2(0,1)), \\
& \Psi_m \to \Psi \quad \text{strongly in} \ L^2(0,T;L^2(0,1)), \\
\end{align*}
where $ \left( P, N ,\Psi \right) $ is a weak solution of $(\mathcal{P})$ in the sense of Definition~\ref{solfaible}.
\end{theo}
As it is well known in the finite volume framework (see for instance \cite{EGH00}), the proof of Theorem \ref{THEO} will be based on estimates satisfied by the approximate solutions and on compactness results.  Due to the particular boundary conditions, we also need some additional results on the convergence of traces. They will be obtained following the ideas of \cite{BrCaHi13}. 

The paper is organized as follows. Section 2 is devoted to the proof of Proposition \ref{PROP} and also to the  proof of discrete $L^2(0,T,H^1)$-estimates on the approximate densities and on the approximate potential.  Then, in Section 3,  we establish the compactness of the sequences of  approximate solutions. We also obtain a convergence result for  the traces on the boundaries. In Section 4 we conclude the proof of Theorem \ref{THEO} by passing to the limit in the numerical scheme. Finally, Section 5 is devoted to the presentation of some numerical experiments.


\section{Existence result and discrete $L^2(0,T,H^1)$-estimates} \label{sectionexistence}


\subsection{Proof of Proposition \ref{PROP}} \label{sectionpreuveprop}

The proof of Proposition \ref{PROP} for $\varepsilon >0$ has already been  done in \cite{BBC12}. Then, we only prove the result for $\varepsilon=0$. To this end, we follow the ideas of \cite{BBC12} and  \cite{BCV14}. 

Letting $\mu >0$, we introduce a mapping $\mathcal{T}^k_{\mu} : \R^{I+2} \times \R^{I+2}  \to  \R^{I+2} \times \R^{I+2} $ such that $\mathcal{T}^k_{\mu}(P,N)=(\bhat{P},\bhat{N})$. This mapping is based on a linearization of the scheme ; it is defined in two successive step. First, we compute $\Psi$ as the solution to the linear system 
\[ - \lambda^2 \left( \dpsi{i+}{} - \dpsi{i-}{} \right) = h_i \left( 3P_i - N_i + \rho_{hl} \right),  \quad \forall i\in \llbracket 1;I\rrbracket, \]
with a definition of the numerical fluxes $\mathrm{d} \Psi_{i + \frac{1}{2}}$ analog to \eqref{fluxpsinum} and  boundary conditions similar to \eqref{CBPsi0}-\eqref{CBPsi1}. The matrix of the linear system is obviously invertible and $\Psi$ is uniquely defined.
%

Then, we define $\bhat{P}$ and $\bhat{N}$ as the solution to the linear systems
\[ \dfrac{h_i}{\Delta t} \left( \left( 1 + \dfrac{\mu}{\lambda^2} \right) \bhat{P}_i - \dfrac{\mu}{\lambda^2} P_i - P_i^k \right) + \mathcal{F}_{P,\, i + \frac{1}{2}} -\mathcal{F}_{P,\, i - \frac{1}{2}}  = 0, \qquad \forall i\in \llbracket 1;I\rrbracket,\]
\[ \dfrac{h_i}{\Delta t} \, \dfrac{\mu}{\lambda^2} \left( \bhat{N}_i - N_i \right) + \mathcal{F}_{N,\, i + \frac{1}{2}} -\mathcal{F}_{N,\, i - \frac{1}{2}}  = 0, \qquad \forall i\in \llbracket 1;I\rrbracket, \]
with a definition of the numerical fluxes $\mathcal{F}_{u,\, i + \frac{1}{2}} $ analog to \eqref{fluxunum} and  boundary conditions similar to \eqref{CBFluxu0}-\eqref{CBFluxu1}. As shown for instance in \cite{BBC12}, the matrices defined at this step are M-matrices. Therefore, they are  invertible and, as in \cite{BBC12},  we can deduce that $\mathcal{T}^k_{\mu}$ preserves the set
\begin{align*}
\mathcal{K} = \left\{ \left( \mathbf{P}, \, \mathbf{N}  \right) \in \R^{I+2} \times \R^{I+2}; \quad 0 \leqslant P_i \leqslant P^{max}, \; 0 \leqslant N_i \leqslant N^{max}, \; \forall 0 \leqslant i \leqslant I+1 \right\},
\end{align*}
as long as (\ref{Dpsi0})-(\ref{Dpsi1}) are satisfied and $\Delta t$ verifies:
\begin{equation}\label{cdtmu}
\Delta t \leqslant \mu \min \left( \dfrac{1}{9P^{max}}, \, \dfrac{1}{N^{max}} \right) .
\end{equation}



Finally, $\mathcal{T}^k_{\mu} $ is a continuous mapping from $\R^{I+2} \times \R^{I+2}$ to itself which preserves the set $\mathcal{K}$. Thanks to Brouwer's Theorem, we conclude that  $\mathcal{T}^k_{\mu} $ has a fixed point in $ \mathcal{K} $. This fixed point with the corresponding $\Psi$ defines a solution to $(\mathcal{S})$ with $\varepsilon=0$ at time step $k+1$. Since $\mu$ is an arbitrary constant, we can choose it such that \eqref{cdtmu} is verified and a solution to the scheme $(\mathcal{S})$  satisfies (\ref{estiLinfu}) without any condition on $\Delta t$.


\subsection{Notations and preliminary results}\label{sectionnorm}

In order to prove Theorem \ref{THEO}, we need to define some functional sets and norms and to establish
some properties. This is the goal of this section.

First of all, let us define two sets of functions.
\begin{defi}
Let $ \mathcal{T}$ be a mesh of  size  $h$ and $\Delta t$ a discretization time step. 
We first define $\Htau$ the set of piecewise constant functions in space as
\begin{multline}\label{set1}
\Htau=\left\{w_h: [0,1] \mapsto \R ~|~ \exists (w_i)_{0\leq i\leq I+1}\in \R^{I+2} \mbox{ such that } \right.\\
\left. w_h(x)=\sum_{i=1}^I w_i\mathds{1}_{\left(x_{i-1/2},x_{i+1/2}\right)}(x)+w_0\mathds{1}_{\{x=0\}}(x)+w_{I+1}\mathds{1}_{\{x=1\}} (x)\right\},
\end{multline}
Then, we define the set of piecewise constant functions in space and time as
\begin{multline}\label{set2}
\Htaudelta=\left\{w_{h,\Delta t} : [0,1] \times [0,T] \mapsto \R ~|~ \exists (w_h^{k+1})_{0\leq k\leq K-1}\in (\Htau)^K \mbox{ such that } \right.\\
\left.
w_{h,\Delta t}(x,t)=\sum_{k=0}^{K-1} w_{h}^{k+1}(x) \mathds{1}_{[t^k,t^{k+1})}(t) \right\}.
\end{multline}
\end{defi}

Then, denoting by $\Vert\cdot\Vert_0$ the usual $L^2(0,1)$-norm, we remark that 
$$
\Vert w_h\Vert_0= \left({\si{1}{I} h_i w_i^2}\right)^{1/2} \quad \forall w_h\in \Htau.
$$
Moreover, we define some norms on $\Htau$ and $\Htaudelta$, which are discrete counterparts of $H^1(0,1)$, $H^{-1}(0,1)$ and $L^2(0,T,H^1(0,1))$-norms: 
\begin{align*}
&\normeun{w_h} = \left( \si{0}{I} \dfrac{\left( w_{i+1} - w_i \right)^2}{\h{i+}} + w_0^2 + w_{I+1}^2 \right)^{\frac{1}{2}}\quad \forall w_h\in \Htau, \\
&\normeduale{w_h} = \max \left\{ \int_0^1 w_h v_h \, \dx, \; v_h \in \Htau \ \text{and} \ \normeun{v_h} \leqslant 1 \right\}\quad \forall w_h\in\Htau,\\
&\normeldeuxhun{w_{h,\Delta t}} = \left( \sn{0}{K} \Delta t\normeun{w_h^{k+1}}^2 \right)^{\frac{1}{2}}\quad 
\forall w_{h,\Delta t}\in \Htaudelta.
\end{align*}


As shown in \cite{BC08}, for all $w_h\in\Htau$, we have:
\begin{align} \label{estifonc1}
\left( w_i \right)^2 \leqslant 2 \normeun{w_h}^2\qquad \forall i\in \llbracket 1;I\rrbracket,
\end{align}
and, as a direct consequence, the following discrete Poincar\'e inequalities:
\begin{align} \label{Poincare}
\normezero{w_h} \leqslant \sqrt{2} \normeun{w_h}\quad \forall w_h\in\Htau.
\end{align}

Finally, we end this section with some properties satisfied by the functional spaces. These properties will be crucial in order to apply some compactness results. More precisely, the following lemmas state that the  hypotheses $(H1)$ and $(H2)$ of Lemma 3.1 in \cite{GL12} hold for sequences $(\Htaum)_m$ and $(w_{h_m})_m$ such that for all $m$, $w_{h_m} \in \Htaum$.

\begin{lem}\label{hyp1gl}
Let $ (\Htaum)_m $ be a sequence of finite-dimensional subspaces of $ L^2(0,1) $ defined by \eqref{set1}.
Let $ ( w_{h_m})_{m}$ be a sequence such that $w_{h_m}\in\Htaum$ for all $m$ and satisfying :
\begin{align*}
\exists C>0 \mbox{ such that }\forall m,\ \normeunm{w_{h_m}} \leqslant C.
\end {align*}
 Then, up to a subsequence, $ ( w_{h_m})_{m} $ converges to $ w_h $ in $ L^2(0,1) $ when $m$ tends to $+\infty$.
\end{lem}
The proof of this lemma is a consequence of a Kolmogorov's compactness Theorem (see for instance Theorem 10.3 in \cite{EGH00}).

\begin{lem}\label{hyp2gl}
Let $ (\Htaum)_m $ be a sequence of finite-dimensional subspaces of $ L^2(0,1) $ defined by \eqref{set1}.
Let $ ( w_{h_m})_{m}$ be a sequence such that $w_{h_m}\in\Htaum$ for all $m$. If $(\fhm{w} )_{m}$ converges to $ w $ in $ L^2(0,1) $ and $ \Vert\fhm{w}\Vert_{-1,2,{\mathcal T}_m} $ converges to $ 0 $, then $ w =0 $.
\end{lem}
\begin{proof}
We will obtain that $w=0$, as a consequence of:
$$
\forall \varphi \in \mathcal{C}_c^{\infty}([0,1]), \quad \int_0^1 w \varphi \ \dx = 0.
$$
Thus, letting  $ \varphi\in \mathcal{C}_c^{\infty}([0, 1]) $, we set $\varphi_i=\varphi(x_i)$ for all $i\in \llbracket 0;I+1\rrbracket$ and we define the associate $\varphi_{h_m}\in \Htaum$. 
Since $ \fhm{w} $ and $ \fhm{\varphi} \in \Htaum $:
\begin{equation}
\label{set3}
\left| \int_0^1 \fhm{w} \fhm{\varphi} \, \dx \right| \leqslant \normedualem{\fhm{w}} \normeunm{\fhm{\varphi}}.
\end{equation}
But, thanks to the regularity of $\varphi$, 
\begin{align*}
\left| \int_0^1 \fhm{w} \varphi \, \dx \right| & \leqslant \left| \int_0^1 \fhm{w} \fhm{\varphi} \, \dx \right| + \left| \int_0^1 \fhm{w}(\varphi - \fhm{\varphi}) \, \dx \right| \\
& \leqslant \normedualem{\fhm{w}} \normeunm{\fhm{\varphi}} + \Vert \fhm{w}\Vert_0 C_{\varphi} h_m.
\end{align*}
Since $ \Vert \fhm{w} \Vert_0 $ is bounded independently of $ h_m $, for all $ \varphi \in \mathcal{C}_c^{\infty}([0,1]) $, we have
\begin{align*}
\int_0^1 w \varphi \, \dx = \lim_{m \to + \infty} \int_0^1 \fhm{w} \varphi \, \dx = 0,
\end{align*}
and then $ w = 0 $.
\end{proof}


\subsection{Discrete $L^2(0,T,H^1(0,1))$ estimates on $P$, $N$ and $\Psi$} \label{sectionestimates}

In order to apply compactness results, we need some estimates on $P_{h,\Delta t}, N_{h, \Delta t}$ and $\Psi_{h, \Delta t}$. Let us begin with the estimate on $\Psi_{h, \Delta t}$.
\begin{prop}\label{estimationPN}
Under the assumptions of Proposition \ref{PROP}, there exists a constant $ C $ depending only on the data and independent of $ \Delta t $ and $ h $, such that:
\begin{equation}
\normeun{\Psi_h^{k+1}}^2 \leqslant C, \quad \forall \, k \geq 0 \label{bornehunpsi},
\end{equation}
\begin{equation}
\mbox{ and }\normeldeuxhun{\Psi_{h,\Delta t}}^2 \leqslant CT \label{borneldeuxpsi}.
\end{equation}
\end{prop}

The proof is left to the reader. It follows the line of the proof of Proposition 2 in \cite{CV13}. It mainly uses \eqref{estiLinfu} and \eqref{Poincare}.

\begin{remark}\label{estilinfpsi}
Thanks to estimates (\ref{estifonc1}) and (\ref{bornehunpsi}), we have $ \normelinftyzeroun{\Psi_h^{k+1}} \leqslant C $ for all $ k \geq 0 $.
\end{remark}
Let us now state the same result on $P_{h,\Delta t}$ and $N_{h,\Delta t}$.
\begin{prop}\label{prop22}
Under the assumptions of Proposition \ref{PROP}, there exists a constant $ C $ depending only on the data and independent of $ \Delta t $ and $ h $, such that:
\begin{align}
\normeldeuxhun{u_{h,\Delta t}}^2 \leqslant C, \quad \mbox{ for } u=P, N \label{estimationldeuxu}
\end{align}
\end{prop}

\begin{proof}
The proof is based on a classical method, already applied in \cite{EGH00}. However, it requires to pay a special attention on the boundary conditions which induce a new difficulty. 

Multiplying (\ref{equnum}) with $ \Delta t u_i^{k+1} $ and summing over $ i $ and $ k $, we obtain $A+B=0$ with 
$$
A= \sn{0}{K} \si{1}{I} \varepsilon_u h_i u_i^{k+1} \left( u_i^{k+1} - u_i^k \right) \mbox{ and } B= \sn{0}{K} \si{1}{I} \Delta t u_i^{k+1} \left( \flux{i+}{k+1} - \flux{i-}{k+1} \right).
$$
It is easy to see that:
\begin{align}
A& \geqslant \sn{0}{K} \si{1}{I} \dfrac{\varepsilon_u h_i}{2} \left( u_i^{k+1} - u_i^k \right)^2 - \si{1}{I} \dfrac{\varepsilon_u h_i}{2} \left( u_i^0 \right)^2 \label{estimationa}.
\end{align}
Moreover, applying a discrete integration by parts to $B$ and using the following decomposition of the numerical fluxes given in \cite{Bes12}:
\begin{multline}\label{decompoflux}
\flux{i+}{k+1}=-z_u \dpsi{i+}{k+1} \dfrac{u_i^{k+1} + u_{i+1}^{k+1}}{2}\\
 + \dfrac{z_u \dpsi{i+}{k+1}}{2} \coth \left( \dfrac{-z_u \h{i+} \dpsi{i+}{k+1}}{2} \right) \left( u_{i+1}^{k+1} -u_i^{k+1} \right),
\end{multline}
we can rewrite $B$ as $B=B_1+B_2+B_3$ with 
\begin{align*}
B_1 & = - \sn{0}{K} \si{0}{I} \dfrac{\Delta t z_u}{2} \dpsi{i+}{k+1} \coth \left( \dfrac{-z_u \h{i+} \dpsi{i+}{k+1}}{2} \right) \left( u_{i+1}^{k+1} -u_i^{k+1} \right)^2, \\
B_2 & = \sn{0}{K} \si{0}{I} \dfrac{\Delta t z_u}{2} \dpsi{i+}{k+1} \left( \left( u_{i+1}^{k+1} \right)^2 - \left( u_i^{k+1} \right)^2 \right), \\
B_3& = \sn{0}{K} \Delta t \left( u_{I+1}^{k+1} \flux{I+}{k+1} - u_0^{k+1} \flux{}{k+1} \right).
\end{align*}

As $x\coth(x)\geq 1$ for all $x\in\R$, we have, as in \cite{Bes12},
\begin{align}\label{estimationb1}
B_1 \geqslant \sn{0}{K} \si{0}{I} \dfrac{\Delta t }{\h{i+}} \left( u_{i+1}^{k+1} - u_i^{k+1} \right)^2.
\end{align}

Applying a discrete integration by parts and using the scheme (\ref{scheme}), we get:
\begin{align*}
B_2 = & \sn{0}{K} \si{1}{I} \dfrac{\Delta t z_u h_i}{2} \dfrac{z_P \left( P_i^{k+1} - P^{max} \right) + z_N \left( N_i^{k+1} -N^{max} \right)}{\lambda^2} \left( u_i^{k+1} \right)^2 \\
& + \sn{0}{K} \dfrac{\Delta t z_u}{2} \left( \dpsi{I+}{k+1} \left( u_{I+1}^{k+1} \right)^2 - \dpsi{}{k+1} \left( u_0^{k+1} \right)^2 \right).
\end{align*}
Since $z_P z_N \left( u_i^{k+1} -u^{max} \right) \left( u_i^{k+1} \right)^2 \geqslant 0$ for $u=N, P$, we have:
\begin{multline*}
\sn{0}{K} \si{1}{I} \dfrac{\Delta t z_u h_i}{2}  \dfrac{z_P \left( P_i^{k+1} - P^{max} \right) + z_N \left( N_i^{k+1} -N^{max} \right)}{\lambda^2} \left( u_i^{k+1} \right)^2 \\
\geqslant \sn{0}{K} \si{1}{I} \dfrac{\Delta t z_u^2 h_i}{2\lambda^2} \left( u_i^{k+1} - u^{max} \right) \left( u_i^{k+1} \right)^2,
\end{multline*}
which yields
\begin{align}\label{estimationb2}
B_2 \geqslant & \sn{0}{K} \si{1}{I} \dfrac{\Delta t z_u^2 h_i}{2\lambda^2} \left( u_i^{k+1} - u^{max} \right) \left( u_i^{k+1} \right)^2  + B_4,
\end{align}
with 
$$
B_4=\sn{0}{K} \dfrac{\Delta t z_u}{2} \left( \dpsi{I+}{k+1} \left( u_{I+1}^{k+1} \right)^2 - \dpsi{}{k+1} \left( u_0^{k+1} \right)^2 \right).
$$
%
Using the boundary conditions \eqref{CBdisc}, we may now rewrite $B_3+B_4$ as 
$$
B_3+B_4= - \sn{0}{K} \Delta t ((f_u^0)^{k+1} + (f_u^1)^{k+1}),
$$
with 
\begin{align*}
(f_u^0)^{k+1} & = \left( u_0^{k+1} \right)^2 \left[ -\beta_u^0 \left( \psi_0^{k+1} \right) + \dfrac{z_u}{2} \, \dfrac{\psi_0^{k+1} - \Delta \psi_0^{pzc}}{\alpha_0} \right] + u_0^{k+1} \gamma_u^0 \left( \psi_0^{k+1} \right), \\
(f_u^1)^{k+1} & = \left( u_{I+1}^{k+1} \right)^2 \left[ -\beta_u^1 \left( V - \psi_{I+1}^{k+1} \right) - \dfrac{z_u}{2} \, \dfrac{ V - \psi_{I+1}^{k+1} - \Delta \psi_1^{pzc}}{\alpha_1} \right] + u_{I+1}^{k+1} \gamma_u^1 \left(V - \psi_{I+1}^{k+1} \right).
\end{align*}
%
%
It remains to find an upper bound of $ (f_u^0)^{k+1} + (f_u^1)^{k+1} $.
To this end, we proceed as in \cite{BBC12} and \cite{CV13}. We introduce:
\begin{align*}
\xi_u^0 (x) & = \gamma_u^0 (x) - u^{max} \beta_u^0 (x) + u^{max} \dfrac{z_u}{\alpha_0}  \left( x - \Delta \psi_0^{pzc} \right), \qquad \forall x \in \R, \\
\xi_u^1 (x) & = \gamma_u^1 (x) - u^{max} \beta_u^1 (x) - u^{max} \dfrac{z_u}{\alpha_1}  \left( x - \Delta \psi_1^{pzc} \right), \qquad \forall x \in \R,
\end{align*}
which are nonpositive functions under the hypotheses \eqref{Dpsi0}-\eqref{Dpsi1} (see \cite{BBC12}).
As
\begin{align*}
(f_u^0)^{k+1} & = \dfrac{\left( u_0^{k+1} \right)^2}{2 u^{max}} \left[ \xi_u^0 \left( \psi_0^{k+1} \right) - u^{max} \beta_u^0 \left( \psi_0^{k+1} \right) - \gamma_u^0 \left( \psi_0^{k+1} \right) \right] + \gamma_u^0 \left( \psi_0^{k+1} \right) u_0^{k+1},
\end{align*}
we clearly have: $(f_u^0)^{k+1}\leq  \gamma_u^0 \left( \psi_0^{k+1} \right) u_0^{k+1}$. 
 Rewriting $(f_u^1)^{k+1}$ with the help of $\xi_u^1$, we similarly prove: $(f_u^1)^{k+1} \leq \gamma_u^1 \left( V -  \psi_{I+1}^{k+1} \right) u_{I+1}^{k+1}$.
%
Then, using the estimates (\ref{estiLinfu}) and (\ref{bornehunpsi}) and the continuity of the functions $ \gamma_u^0 $ and $ \gamma_u^1 $, we have 
\begin{align}\label{estimationb3b4}
B_3+B_4\geq -C.
\end{align}
 From \eqref{estimationa}, \eqref{estimationb1}, \eqref{estimationb2} and \eqref{estimationb3b4}, we deduce that 
\begin{align} \label{estimationu}
\sn{0}{K} \si{0}{I} \Delta t \dfrac{\left( u_{i+1}^{k+1} - u_i^{k+1} \right)^2}{\h{i+}} + \dfrac{\varepsilon_u}{2} \sn{0}{K} \si{1}{I} h_i \left( u_i^{k+1} - u_i^k \right)^2 \leqslant C,
\end{align}
with $ C $ depending only on $ P^{max} $, $ N^{max} $, $ \alpha_0 $, $ \alpha_1 $, $ V $, $ \Dpsi{0} $, $ \Dpsi{1} $, $ \lambda $, $ \left( \beta_u^{i}, \gamma_u^{i} \right)_{i=0, \, 1} $ and $ T $. This ends the proof of Proposition \ref{prop22}
\end{proof}

\begin{remark}\label{estitempsu}
Note that a direct consequence of (\ref{estimationu}) is
\begin{align}\label{esttempsu}
\varepsilon_u \sn{0}{K} \si{1}{I} h_i \left( u_i^{k+1} - u_i^k \right)^2 \leqslant C,
\end{align}
with $C$ a constant independent of $h$ and $\Delta t$. This inequality will be used in Section~\ref{compacitepsi}.
\end{remark}


\section{Compactness results and passage to the limit} \label{sectioncompacite}

In this section, we prove the Theorem \ref{THEO}. Firstly, we establish the convergences of $ (P_m)_m $ and $ (N_m)_m $ using a Gallou\"et-Latch\'e compactness Theorem (Theorem 3.4 in \cite{GL12}), which is  a discrete counterpart of the Aubin-Simon Lemma. Secondly, we prove the convergence of $ (\Psi_m)_m $ using a classical Kolmogorov Theorem. Then, we show the convergence of the traces on the boundaries following the ideas of \cite{BrCaHi13}. Finally, passing to the limit in the scheme, we  prove that $(P_m, N_m, \Psi_m)_m$ tends to a solution of $(\mathcal{P})$ in the sense of Definition \ref{solfaible}.


\subsection{Compactness of $ (P_m)_m $ and $ (N_m)_m $}

The proof of compactness being analogous for $(P_m)_m$ and $(N_m)_m$, in all this section we use the notation $(u_m)_m$ where $u$ can be replaced by $P$ or $N$. To prove the compactness of the sequence $(u_m)_m$, we will use Theorem 3.4 in \cite{GL12}. 
Therefore, for any function $ w_{h,\Delta t} \in \Htaudelta $, we define a discrete time derivative $\partial_{t,\mathcal{T}} w_{h,\Delta t}$
and a discrete space derivative $\partial_{x,\mathcal{T}} w_{h, \Delta t}$. There are piecewise constant function in space and time, defined by:
\begin{gather*}
\partial_{t,\mathcal{T}} w_{h,\Delta t}(x,t)=\partial_{t,\mathcal{T}}^k w_{h,\Delta t}= \dfrac{1}{\Delta t }(w_h^{k+1}- w_h^k)~~\hbox{on}~[t_k,t_{k+1}),\\
\partial_{x,\mathcal{T}} w_{h, \Delta t}(x,t)= \partial_{x,\mathcal{T}}^i w_{h, \Delta t}= \dfrac{w_{i+1}^{k+1}-w_i^{k+1}}{h_{i+1/2} }, ~~\hbox{on}~ (x_i,x_{i+1})\times (t_k,t_{k+1}).
\end{gather*}


Let us first establish an estimate on the discrete time derivative needed for the convergence proof. 
\begin{prop}\label{propestimationderivePN}
Under the assumptions of Theorem \ref{THEO}, there exists a constant $ C $ depending only on the data such that:
\begin{equation}\label{estu:temp}
\sum_{k=0}^{K_m-1} \Delta t_m\normedualem{\partial_{t,\mathcal{T}_m}^k u_m }^2 \;  \leqslant C.
\end{equation}
\end{prop}
\begin{proof}
Let consider $ v_{h_m} \in \Htaum $ such that  $ \normeunm{v_{h_m}} \leqslant 1 $. By definition, we have:
$$
\int_0^1 \left( \partial_{t,\mathcal{T}_m}^ku_m \right) v_{h_m} =\sum_{i=0}^I h_i \frac{u_i^{k+1}-u_i^k}{\Delta t} v_i.
$$
Then, using  the scheme \eqref{equnum}, a discrete integration by parts and the reformulation \eqref{decompoflux} of the fluxes, we get 
\begin{equation*}
\left| \int_0^1 \left( \partial_{t,\mathcal{T}_m}^ku_m \right) v_{h_m} \, \dx \right| \leqslant A_1 + A_2 + A_3,
\end{equation*}
with
\begin{align*}
& A_1 = \dfrac{1}{\varepsilon_u} \si{0}{I} \left| v_{i+1} - v_i \right| \left| z_u \dpsi{i+}{k+1} \dfrac{u_{i+1}^{k+1}+u_i^{k+1}}{2} \right|, \\
& A_2 = \dfrac{1}{\varepsilon_u} \si{0}{I} \left| v_{i+1} - v_i \right| \left| \dfrac{z_u \dpsi{i+}{k+1}}{2} \coth \left( \dfrac{- z_u \h{i+} \dpsi{i+}{k+1}}{2} \right) \left( u_{i+1}^{k+1} - u_i^{k+1} \right) \right|, \\
& A_3 = \dfrac{1}{\varepsilon_u} \left| \flux{}{k+1} v_0 \right| + \dfrac{1}{\varepsilon_u} \left| \flux{I+}{k+1} v_{I+1} \right|.
\end{align*}
Using (\ref{estiLinfu}) and (\ref{bornehunpsi}), we obtain:
\begin{align*}
A_1 & \leqslant \dfrac{\vert z_u\vert u^{max}}{\varepsilon_u} \left( \si{0}{I} \dfrac{ \left(v_{i+1} - v_i \right)^2}{\h{i+}} \right)^{\frac{1}{2}} \left( \si{0}{I} \dfrac{ \left( \Psi_{i+1}^{k+1} - \Psi_{i}^{k+1} \right)^2 }{\h{i+}} \right)^{\frac{1}{2}} \leqslant \dfrac{C}{\varepsilon_u} \normeunm{v_{h_m}}.
\end{align*}
Since $ x \mapsto x \coth (x) $ is a 1-Lipschitz continuous function and is equal to 1 in 0 and thanks to Cauchy-Schwarz inequality, we obtain: 
\begin{align*}
A_2 & \leq \dfrac{1}{\varepsilon_u} \si{0}{I} \dfrac{ \left| v_{i+1} - v_i \right| }{\h{i+}} \left| \dfrac{z_u \h{I+} \dpsi{i+}{k+1}}{2} \coth \left( \dfrac{- z_u \h{i} \dpsi{i+}{k+1}}{2} \right) - 1 \right| \left| u_{i+1}^{k+1} - u_i^k \right| \\
& \phantom{\leq} + \dfrac{1}{\varepsilon_u} \si{0}{I} \dfrac{\left| v_{i+1} - v_i \right|}{\h{i+}} \left| u_{i+1}^{k+1} - u_i^k \right| \\
& \leq \left( \dfrac{C}{\varepsilon_u} \normeunm{\Psi_m^{k+1}} + \dfrac{1}{\varepsilon_u} \normeunm{u_m^{k+1}} \right) \normeunm{v_{h_m}}.
\end{align*}
Let us now consider the term $A_3$ containing the boundary conditions. Using \eqref{bornehunpsi} and the continuity of the functions $ (\beta_u^{i}, \gamma_u^{i})_{i = 0, \, 1} $, we have
\begin{align*}
\varepsilon_{u}A_3 & \leqslant \left( \gamma_u^0 \left( \left| \Psi_0^{k+1} \right| \right) + \left| u_0^{k+1} \right| \beta_u^0 \left( \left| \Psi_0^{k+1} \right| \right) \right) \left| v_0 \right| \\
&\qquad+\left( \gamma_u^1 \left( \left| V - \Psi_{I+1}^{k+1} \right| \right) + \left| u_{I+1}^{k+1} \right| \beta_u^1 \left( \left| V - \Psi_{I+1}^{k+1} \right| \right) \right) \left| v_{I+1} \right| \\
& \leqslant C \left(1+| u_0^{k+1}|+\left| u_{I+1}^{k+1} \right| \right)\left( \left| v_0 \right| + \left| v_{I+1} \right| \right) \\
& \leqslant C (1+\normeunm{u_{m}^{k+1}} )\normeunm{v_{h_m}}.
\end{align*}
Therefore, we obtain:
\begin{align*}
\left| \int_0^1 \left( \partial_{t,\mathcal{T}_m}^k u_m \right) v_{h_m} \, \dx \right| & \leqslant \dfrac{C\normeunm{v_{h_m}}}{\varepsilon_u} \left( 1 +\normeunm{u_{m}^{k+1}} \right),
\end{align*}
which yields
\begin{align*}
\sum_{k=0}^{K-1} \!  \Delta t_m\normedualem{\partial_{t,\mathcal{T}_m}^k u_m}^2 \,  & \leqslant \dfrac{2C^2 T}{\varepsilon_u^2} + \dfrac{2C^2}{\varepsilon_u^2} \sn{0}{K} \Delta t_m \normeunm{u_{m}^{k+1}}^2.
\end{align*}
Finally, the estimate \eqref{estu:temp} is a consequence of the $L^2(0,T,H^1(0,1))$ estimate  (\ref{estimationldeuxu}).
\end{proof}

We are now  able to prove the following convergence result.
\begin{prop}\label{CVPN}
Under the assumptions of Theorem \ref{THEO}, there exists \break$u \in L^2(0,T;H^1(0,1))$ such that, up to a subsequence, 
\begin{align*}
& u_m \rightarrow u \quad \hbox{ strongly in } L^2(0, T; L^2(0,1)), \mbox{ when $m\to +\infty$}\\
& \partial_{x,\mathcal{T}_m} u_m \rightharpoonup \partial_x u  \quad \hbox{ weakly in } L^2(0, T; L^2(0,1)) \mbox{ when $m\to+\infty$}.
\end{align*}
\end{prop}

\begin{proof}
The sequence $ (\Htaudeltam)_m $ is a sequence of finite-dimensional subspaces of $L^2(0,1)$. Each subspace $\Htaudeltam$ can be endowed either with the norm $\Vert\cdot\Vert_{1,\mathcal{T}_m}$ or with the norm $\Vert \cdot \Vert_{-1,2,\mathcal{T}_m}$. These two norms verify Lemma \ref{hyp1gl} and Lemma \ref{hyp2gl} which correspond to the hypotheses of Theorem 3.4 in \cite{GL12}.

Then, thanks to Proposition \ref{prop22} and \ref{propestimationderivePN}, 
%
%
we can apply Theorem 3.4 in \cite{GL12} : up to a subsequence, $ u_m \rightarrow u $ strongly
 in $ L^2(0, T, L^2(0,1)) $. Moreover, Proposition \ref{prop22} implies that $ u \in L^2(0,T,H^1(0,1)) $. Finally the weak convergence of $\partial_{x,\mathcal{T}_m} u_m$ in $L^2(0,T;L^2(0,1))$ is also a consequence of Proposition \ref{prop22}.
\end{proof}


\subsection{Compactness of $ (\Psi_m)_m $}\label{compacitepsi}

The convergence of the sequence $ (\Psi_m)_m $ will be obtained as a consequence of the Kolmogorov compactness Theorem, as it is done for instance in \cite{EGH00}.

Let us first remark that the following estimate on the space translates of $\Psi_{h,\Delta t}$ is a consequence of the estimate \eqref{borneldeuxpsi}.

\begin{lem}\label{lemtranspsibis}
Under the assumptions of Theorem \ref{THEO}, there exists a constant $ C $ such that for all $\eta < h$:
\begin{equation*}
\left\| \Psi_{h,\Delta t}(\cdot+\eta,\cdot)-\Psi_{h,\Delta t}(\cdot,\cdot)\right\|_{L^2(0,T,L^2(0,1-\eta))}^2 \leq C \eta
\end{equation*}
\end{lem}

Then, we can also establish an estimate on the time translates of $\Psi_{h,\Delta t}$.
\begin{lem}\label{lemtranspsi}
Under the assumptions of Theorem \ref{THEO}, there exists a constant $ C $ such that for all $\tau < \Delta t$:
\begin{equation*}
\left\| \Psi_{h,\Delta t}(.,.+\tau)-\Psi_{h,\Delta t}(.,.)\right\|_{L^2(0,T-\tau,L^2(0,1))}^2 \leq C \tau
\end{equation*}
\end{lem}

\begin{proof}
Using (\ref{eqpsinum}), (\ref{CBPsi0}) and (\ref{CBPsi1}), we have, for $i\in \llbracket 1;I\rrbracket$,
\begin{align*}
\left\{
\begin{array}{l}
-\lambda^{2} \left( \dpsi{i +}{k+2} - \dpsi{i +}{k+1} - \dpsi{i -}{k+2} + \dpsi{i -}{k+1}  \right) =
\\
\hspace{5.cm} h_{i} \left( 3 \left( P_{i}^{k+2} - P_{i}^{k+1} \right) - \left( N_{i}^{k+2} - N_{i}^{k+1} \right) \right), \\
\Psi_{0}^{k+2} - \Psi_{0}^{k+1}  = \alpha_{0} \left( \dpsi{}{k+2} - \dpsi{}{k+1} \right), \\
\Psi_{I+1}^{k+2} - \Psi_{I+1}^{k+1}  = -\alpha_{1} \left( \dpsi{I +}{k+2} - \dpsi{I +}{k+1} \right).
\end{array}
\right.
\end{align*}
Multiplying by $\left( \Psi_i^{k+2} - \Psi_i^{k+1} \right) $ and summing over $ i $ and $ k $, we obtain $A=B$ with
\begin{align*}
& A=\sn{0}{K} \si{1}{I} - \lambda^2  \left( \dpsi{i +}{k+2} - \dpsi{i -}{k+2} - \dpsi{i +}{k+1} + \dpsi{i -}{k+1} \right) \left( \Psi_i^{k+2} - \Psi_i^{k+1} \right), \\
& B= \sn{0}{K} \si{1}{I} h_{i}  \left( 3 \left( P_{i}^{k+2} - P_{i}^{k+1} \right) - \left( N_{i}^{k+2} -N_{i}^{k+1} \right) \right) \left( \Psi_i^{k+2} - \Psi_i^{k+1} \right).
\end{align*}
Using the boundary conditions, in a same way as in the proof of Proposition \ref{estimationPN}, we get:
\begin{align*}
 A\geqslant \lambda^2 \alpha \sn{0}{K}  \normeun{\Psi^{k+2}_{h} - \Psi^{k+1}_{h}}^2,
\end{align*}
with $ \alpha = \min (1, 1/\alpha_0, 1/\alpha_1) $.
Then, using Young's inequality, (\ref{Poincare}) and Remark \ref{estitempsu}, we also have:
\begin{align*}
B \leqslant \dfrac{\lambda^2 \alpha}{2} \sn{0}{K}  \normeun{\Psi_{h}^{k+2} - \Psi_{h}^{k+1}}^2 + C.
\end{align*}
Then, we obtain
\begin{align*}
\sn{0}{K}  \normeun{\Psi_{h}^{k+2} - \Psi_{h}^{k+1}}^2 \leqslant C,
\end{align*}
with $C$ a constant independent of $\tau$ and $\Delta t$. With \eqref{Poincare}, this yields the expected result, since
$$\left\| \Psi_{h,\Delta t}(.,.+\tau)-\Psi_{h,\Delta t}(.,.)\right\|_{L^2(0,T-\tau,L^2(0,1))}^2=\tau \sum_{k=0}^{K-1} \normezero{\Psi_h^{k+2}-\Psi_h^{k+1}}^2.$$
\end{proof}

Finally, we deduce from Lemmas \ref{lemtranspsibis} and \ref{lemtranspsi} the following convergence result for the sequence $(\Psi_m)_m$. 

\begin{prop}\label{CVpsi}
Under the assumptions of Theorem \ref{THEO},  there exists \break$\Psi \in L^2(0,T;H^1(0,1))$ such that, up to a subsequence,
\begin{align*}
& \Psi_m \rightarrow \Psi \quad \hbox{ strongly in } L^2(0, T; L^2(0,1)), \mbox{ when $m\to +\infty$,}\\
& \partial_{x,\mathcal{T}_m} \Psi_m \rightharpoonup \partial_x \Psi  \quad \hbox{ weakly in } L^2(0, T; L^2(0,1)),\mbox{ when $m\to +\infty$}.
\end{align*}
\end{prop}

%
%

\subsection{Convergence of the traces}\label{traces}

Up to now, we have proved the existence of $P$, $N$ and $\Psi$ belonging to $L^2(0,T;H^1(0,1)) $, such that, up to a subsequence, $(P_m)_m$, $(N_m)_m$ and $(\Psi_m)_m$ converge respectively to $P$, $N$ and $\Psi$. It remains to prove that $(P,N,\Psi)$ is a solution to the corrosion model in the sense of Definition  \ref{solfaible}. Therefore, we will pass to the limit in the scheme, see Section \ref{sectionlimite}. But, at this stage, we will have to pass to the limit in some boundary terms. Therefore, we need the convergence of different traces. 
%
%

For $w=N,P,$ or $\Psi$, we have established that $w\in L^2(0,T;H^1(0,1)) $. Due to the compact embedding of $H^1(0,1)$ into $\mathcal{C}([0,1])$ in 1D, it is clear that the trace of $w$ in 0 and 1 is equal respectively  to $w(0,\cdot)$ and $w(1,\cdot)$, which belong to $L^2(0,T)$.

For a given function $w_{h,\Delta t}\in {\mathcal H}_{{\mathcal T},\Delta t}$, the usual trace, denoted by $\gamma w_{h,\Delta t}$ is a piecewise constant function of time defined by 
$$
\gamma w_{h,\Delta t}(0,t) = w_1^{k+1}\quad {\rm and} \quad\gamma w_{h,\Delta t}(1,t) = w_{I}^{k+1}, \quad \forall t\in [t^k, t^{k+1}).
$$
But, it will be easier to deal with an approximate trace ${\tilde \gamma} w_{h,\Delta t}$ defined by 
$$
\tild{\gamma} w_{h,\Delta t}(0,t) = w_0^{k+1}\quad {\rm and} \quad
\tild{\gamma} w_{h,\Delta t}(1,t) = w_{I+1}^{k+1}, \quad \forall t\in [t^k, t^{k+1}) .
$$


\begin{prop}
For $w=N,P,$ or $\Psi$, the sequence $ (\tild{\gamma} w_m(0,\cdot))_m $ (resp. $ (\tild{\gamma} w_m(1,\cdot))_m $) converges towards $  w(0,\cdot) $ (resp. $ w(1,\cdot) $) strongly in $ L^1(0,T) $, up to a subsequence.
\end{prop}

\begin{proof}
To prove this Proposition we follow the ideas presented in the proof of lemma 4.8 of \cite{BrCaHi13}.  More precisely , for all $\varrho >0$, we write:
\begin{align}
\intt{0}{T}{\left| \tild{\gamma}w_m (0,t) - w(0,t) \right|} & = \dfrac{1}{\varrho} \intm{0}{\varrho}{\intt{0}{T}{\left| \tild{\gamma}w_m (0,t) -  w(0,t) \right|}}{y} \leqslant T_1 + T_2 + T_3, \label{T123}
\end{align}
with
\begin{align*}
& T_1 = \dfrac{1}{\varrho} \intm{0}{\varrho}{\intt{0}{T}{\left| \tild{\gamma}w_m (0,t) - w_m (y,t) \right|}}{y}, \\
&T_2 = \dfrac{1}{\varrho} \intm{0}{\varrho}{\intt{0}{T}{\left| w_m (y,t) - w(y,t) \right|}}{y}, \\
& T_3 = \dfrac{1}{\varrho} \intm{0}{\varrho}{\intt{0}{T}{\left| w(y,t) -  w(0,t) \right|}}{y}. 
\end{align*} 
Let us first note that $T_3\to 0$ when $\varrho \to 0$, as $w(0,t)$ is the trace on $x=0$ of $w(\cdot,t)$. Let us now show that $T_1$ and $T_2$ converge also to 0.

Let us define $ K_{\varrho}=\left\lfloor \dfrac{\varrho}{h_m} \right\rfloor + 1 $. We can show:
\begin{align*}
T_1 & \leqslant \frac{1}{\varrho} \sn{0}{K} \Delta t_m \som{1}{K_{\varrho}}{p} h_p \left| w_0^{k+1} - w_p^{k+1} \right|  \leqslant \dfrac{1}{\varrho} \sn{0}{K} \Delta t_m \som{1}{K_{\varrho}}{p} h_p \si{0}{p-1} \left| w_{i+1}^{k+1} - w_i^{k+1} \right|.
\end{align*}
Using Cauchy-Schwarz inequality and \eqref{borneldeuxpsi} or \eqref{estimationldeuxu}
\begin{align}
T_1 & \leqslant C \sqrt{T} \dfrac{(h_m + \varrho)^{\frac{3}{2}}}{\varrho} \underset{h_m \to 0}{\longrightarrow} C \sqrt{T} \varrho^{\frac{1}{2}}. \label{T1}
\end{align}
Moreover,
\begin{align}
T_2 \leqslant \dfrac{\sqrt{T}}{\sqrt{\varrho}} \left( \intm{0}{\varrho}{\intt{0}{T}{ (w_m (y,t) - w(y,t) )^2}}{y} \right)^{\frac{1}{2}} \underset{h_m \to 0}{\longrightarrow} 0\label{T2},
\end{align}
which gives with (\ref{T123}), (\ref{T1}), (\ref{T2}) and $\rho$ tends to 0, 
\begin{align*}
\lim_{h_m \to 0} \intt{0}{T}{\left| \tild\gamma w_m (0,t) - w(0,t) \right|} = 0.
\end{align*}
This concludes the proof.
\end{proof}

\begin{cor} \label{Corotrace}
Up to a subsequence, the sequence $ (\tild{\gamma} w_m(0,\cdot))_m $ (resp. $ (\tild{\gamma} w_m(1,\cdot))_m $) converges towards $  w(0,\cdot) $ (resp. $ w(1,\cdot) $) strongly in $ L^2(0,T) $, for $w=P,N$ or $\Psi$.
\end{cor}

\begin{proof}
The sequence $ (\tild{\gamma} w_m(0,\cdot))_m $ (resp. $ (\tild{\gamma} w_m(1,\cdot))_m $) converges almost everywhere on $ (0,T) $. Moreover, $\tild{\gamma}w_m(0,\cdot) $ (resp. $ \tild{\gamma} w_m(1,\cdot) $) is uniformly bounded thanks to the Propositions \ref{PROP} and \ref{estimationPN}. Then, we obtain the strong convergence in $ L^2(0,T) $.
\end{proof}


\subsection{Passage to the limit} \label{sectionlimite}

We end the proof of Theorem \ref{THEO}. Indeed, it remains to prove that the limit $P, N$, $ \Psi $, defined in Propositions \ref{CVPN} and \ref{CVpsi} is a solution of $(\mathcal{P})$ in the sense of Definition \ref{solfaible}. 
To this end, we follow the method used in \cite{Bes12} and \cite{ChLiPe03}. 

First of all we prove that ($P, N$, $\Psi$) satisfy \eqref{DDufaible}. For $u=P,N$ and $\varphi \in \mathcal{D}([0,1]\times [0,T[)$, we define
\begin{align*}
& A_{10} (m) = - \left( \int_0^T \! \! \! \! \int_0^1 \varepsilon_u u_m(x,t) \left( \partial_t  \varphi(x,t) \right) \; \dx \dt + \int_0^1 \varepsilon_u u_m(x,0) \varphi(x,0) \; \dx \right), \\
& A_{20} (m) = \int_0^T \! \! \! \! \int_0^1 \left( \partial_{x,\mathcal{T}_m} u_m(x,t) \right) \left( \partial_x \varphi(x,t) \right) \; \dx \dt, \\
& A_{30} (m) = \int_0^T \! \! \! \! \int_0^1 z_u u_m(x,t) \left( \partial_{x,\mathcal{T}_m} \Psi_m(x,t) \right) \left( \partial_x \varphi(x,t) \right) \; \dx \dt, \\
& A_{40}(m) = \int_0^T \left( \beta_u^1(V - \tild{\gamma} \Psi_m(1,t))\tild{\gamma} u_m(1,t) - \gamma_u^1(V - \tild{\gamma} \Psi_m(1,t)) \right) \varphi(1,t) \ \dt, \\
& \qquad \qquad + \int_0^T \left( \beta_u^0( \tild{\gamma} \Psi_m(0,t)) \tild{\gamma} u_m(0,t) - \gamma_u^0 (\tild{\gamma} \Psi_m(0,t))\right) \varphi(0,t) \ \dt, \\
& \varepsilon_A (m) = - A_{10}(m) - A_{20}(m) - A_{30}(m) - A_{40}(m).
\end{align*}
Due to Propositions \ref{CVPN}, \ref{CVpsi} and Corollary \ref{Corotrace}, we have:
\begin{align*}
\varepsilon_A (m) & \underset{m \to + \infty}{\longrightarrow}  \int_0^T \! \! \! \! \int_0^1 \varepsilon_u u \left( \partial_t  \varphi \right) \; \dx \dt +\int_0^1 \varepsilon_u u(x,0) \varphi(x,0) \; \dx + \int_0^T \! \! \! \! \int_0^1 J_{u} \left( \partial_x \varphi \right) \; \dx \dt \\
& \quad - \int_0^T \left( \beta_u^1(V - \Psi(1,t))u(1,t) - \gamma_u^1(V - \Psi(1,t)) \right) \varphi_1(1,t) \ \dt \\
& \quad - \int_0^T \left( \beta_u^0(\Psi(0,t))u(0,t) - \gamma_u^0(\Psi(0,t))\right) \varphi_1(0,t) \ \dt.
\end{align*}
Thus, it remains to prove that 
$\varepsilon_A (m) {\longrightarrow} 0$, when $m\to+\infty$.
To this end,  we multiply the scheme (\ref{equnum}) by $ \Delta t \, \varphi_i^k $, with $\varphi_i^k=\varphi(x_i,t^k)$, and sum over $ i $ and $ k $. Using the fluxes decomposition \eqref{decompoflux} and the boundary conditions \eqref{CBdisc}, we obtain:
\begin{align*}
A_1(m)+A_2(m) + A_3(m) + A_4(m)=0, 
\end{align*}
with
\begin{align*}
A_1(m) & = \sn{0}{K} \si{1}{I} \varepsilon_u h_i \left( u_i^{k+1} - u_i^k \right) \varphi_i^k \\
& = -\varepsilon_u \sum_{k=0}^{K-1} \sum_{i=1}^I \int_{t^k}^{t^{k+1}} \int_{x_{i-1/2}}^{x_{i+1/2}} u_i^{k+1} \partial_t \varphi(x_i,t)\;\dx\dt -\varepsilon_u \sum_{i=1}^I \int_{x_{i-1/2}}^{x_{i+1/2}} u_i^0\varphi(x_i,0)\;\dx,\\
A_2(m) & = - \sn{0}{K} \si{0}{I} \Delta t_m\dfrac{z_u \dpsi{i+}{k+1}}{2} \coth \left( \dfrac{-z_u \h{i+} \dpsi{i+}{k+1}}{2} \right) \left(u_{i+1}^{k+1} - u_i^{k+1} \right) \left( \varphi_{i+1}^k - \varphi_i^k \right),\\
A_3(m) & = \sn{0}{K} \si{0}{I} \Delta t_m z_u \dpsi{i+}{k+1} \dfrac{u_{i+1}^{k+1} + u_i^{k+1}}{2} \left( \varphi_{i+1}^k - \varphi_i^k \right), \\
&= z_u\sn{0}{K} \si{0}{I}  \frac{u_{i+1}^{k+1} + u_i^{k+1}}{2} \dpsi{i+}{k+1} \int_{t^k}^{t^{k+1}} \left( \varphi(x_{i+1},t^k)-\varphi(x_{i},t^k)\right) \; \dt,\\
A_4(m) &  = \sn{0}{K} \Delta t_m \left[ \left( \beta_u^1(V - \Psi_{I+1}^{k+1}) u_{I+1}^{k+1} - \gamma_u^1(V - \Psi_{I+1}^{k+1}) \right) \varphi_{I+1}^k \right.\\
&\hskip6cm + \left. \left( \beta_u^0( \Psi_{0}^{k+1}) u_{0}^{k+1} - \gamma_u^0(\Psi_{0}^{k+1}) \right) \varphi_0^k \right].
\end{align*}
 Following the standard method used in \cite{Bes12}, we can prove that
\[ \vert A_{i} - A_{i0} \vert \underset{m \to + \infty}{\longrightarrow} 0, \quad \text{for} \ i \in \llbracket 1, \, 3 \rrbracket. \]
The only difference with the standard method concerns the terms related to the boundary conditions.

Then, let us compare $ A_4 $ with $ A_{40} $. We can rewrite $ A_{40} $ under the form:
\begin{align*}
A_{40}(m) & = \sn{0}{K} \left( \beta_u^1(V - \Psi_{I+1}^{k+1}) u_{I+1}^{k+1} - \gamma_u^1(V - \Psi_{I+1}^{k+1}) \right) \int_{t^k}^{t^{k+1}} \varphi(1,t) \; \dt \\
& \quad - \sn{0}{K} \left( \beta_u^0(\Psi_0^{k+1}) u_0^{k+1} - \gamma_u^0(\Psi_0^{k+1}) \right) \int_{t^k}^{t^{k+1}} \varphi(0,t) \; \dt ,
\end{align*}
and, using the continuity of the functions $ (\beta_u^{i}, \gamma_u^{i})_{i=0, \,1} $, the regularity of $\varphi$, \eqref{estiLinfu} and \eqref{estifonc1}, we obtain:
\begin{align*}
\left| A_{4}(m) - A_{40}(m) \right| & \leqslant C \|\varphi\|_{\mathcal{C}^2((0,T)\times[0,1])}\Delta t_m \underset{m \to + \infty}{\longrightarrow} 0.
\end{align*}
Finally, $ \varepsilon_A(m) \underset{m \to + \infty}{\longrightarrow} 0 $ and $P, N$, $ \Psi $ satisfy \eqref{DDufaible}. In the same way, we can prove that $P, N$ and $ \Psi $ satisfy \eqref{poissonfaible}. This concludes the proof of Theorem \ref{THEO}.

\section{Numerical experiments} \label{sectionexperience}

In this Section, we present some numerical results for a test case close to the real case given in \cite{BBC12}. Indeed, we have modified some data of the real case in order to have a simplified model which is close to the DPCM model.
For the DPCM and simplified models, it was performed a scaling relative to the characteristic time of the cations. This gives a very small value for the coefficient $ \varepsilon $. Therefore, even if we have proved the convergence of the scheme only for $\varepsilon>0$, we want to study numerically the behavior of the scheme for different values of $\varepsilon$. More precisely, we are interested in its behaviour in the limit $\varepsilon\to 0$. 

\subsection{Presentation of the test case}

The test case, introduced here, is inspired by the real case given in \cite{BBC12}. Starting from the data given in \cite{BBC12}, we have built a test case  which is adapted for the simplified model. We only consider a potentiostatic case, which means that $V$ is an applied potential. 
For the different experiments, we will consider the numerical values for the dimensionless simplified model given in Table \ref{parametres}.

\begin{table}[h!]
\caption{Dimensionless parameters.} \label{parametres}
\begin{center}
\begin{tabular}{c|c|c|c|c|c|c|c|c}
$ \lambda^2 $ & $ \alpha_0 $ & $ \alpha_1 $ & $ P^m $ & $ N^m $ & $ k_P^0 $ &  $ k_N^0 $ & $ k_P^1 $ & $ k_N^1 $ \\ 
\hline \hline $ 1.1 \ 10^{-3} $  & $ 0.177 $ & $ 0.089 $ & $ 2 $ & $ 1 $ & $ 10^{8} $ & $ 10^{-18} $ & $ 10^{11} $ & $ 26.8 $ \\ 
\hline \hline  $ m_P^0 $ & $ m_N^0 $ & $ m_P^1 $ & $ m_N^1 $ & $ a_P^0 $ & $ a_N^0 $ & $ b_P^0 $ & $ b_N^0 $ & $ a_P^1 $ \\ 
\hline \hline  $ 0 $ & $ 1.45 \ 10^{-24} $ & $ 10^{8} $ & $ 26.8 $ & $ 0.5 $ & $ 0.5 $ & $ 0.5 $ & $ 0.5 $ & $ 0.5 $ \\ 
\hline \hline $ a_N^1 $ & $ b_P^1 $ & $ b_N^1 $ & $ \Delta \Psi_0^{pzc} $ & $ \Delta \Psi_1^{pzc} $ & $ V $ & $ \rho_{hl} $ & &  \\ 
\hline \hline $ 0.5 $ & $ 0.5 $ & $ 0.5 $ & $ -0.866 $ & $ 0 $ & $ 0.5 $ & $ -5 $ & &  \\ 

\end{tabular} 
\end{center}
\end{table}

Let us remark that for such a choice of coefficients all the hypotheses (\eqref{hyp_mk}, \eqref{hyp_ab}, \eqref{hyp_NPrho}, $\alpha_0$ and $\alpha_1$ positive, \eqref{Dpsi0} and \eqref{Dpsi1}) are satisfied except for the right inequality in \eqref{Dpsi1}. Nevertheless, we observe that numerically \eqref{estiLinfu} still holds. 

As mentioned in the introduction, a scaling relative to the charateristic time gives a very small coefficient $ \varepsilon $, which is the quotient of the mobilities of the densities. In practice, $\varepsilon$ is close to $10^{-14}$.
Then, for the experiments, we will consider different small values of $ \varepsilon $, $ \varepsilon \in \{ 0, 10^{-6}, 10^{-4}, 10^{-2} \} $ in order to verify that the scheme is asymptotic preserving in the limit $\varepsilon$ tends to zero.

\subsection{Numerical results }

First, we want to illustrate numerically the convergence of the scheme, for different values of $\varepsilon$ and its asymptotic preserving behavior in the limit $\varepsilon$ tends to zero. 
Since the exact solution of this problem is not available, we compute a reference solution on a uniform mesh made of $ 4000 $ cells with a time step $ \Delta t = 10^{-6} $ for each values of $ \varepsilon $. For the simplified model, there is a stationary state in long-time. Thus, it is important to take a small time $ T$  in order to compute a solution different from the stationary state. Then, for Figures \ref{erreurl2}, \ref{erreurl2tps} and \ref{erreurl2tpsepsi}, the $ \mathrm{L}^2 $-norms in space are computed at the final time $ T=10^{-3} $.

Figure \ref{erreurl2} shows the $ \mathrm{L}^2 $-convergence rate in space of the scheme for different values of $\varepsilon$. 
We observe  that  the convergence rate is independent of $ \varepsilon $ and the scheme has an order 2 in space, even for $\varepsilon=0$. It is due to the choice of  Scharfetter-Gummel fluxes for the numerical approximation of the convection-diffusion fluxes. Figure \ref{erreurl2tps} shows the $L^2$-convergence rate in time for different values of $\varepsilon$. We observe that this convergence rate is also independent of $\varepsilon$ and the scheme has an order 1 in time as expected. 

In Figure \ref{erreurl2tpsepsi} we present, for different values of the time step, the $ \mathrm{L}^2 $ error in space at the final time $T=10^{-3}$ with respect to $ \varepsilon$. The asymptotic preserving behavior of the scheme, in the limit $\varepsilon$ tends to zero, clearly appears.   

\begin{figure}[h!]
\begin{center}
\begin{subfigure}[b]{0.45\textwidth}
\caption{$ \mathrm{L}^2 $ error on the cations density}
\includegraphics[width=\textwidth]{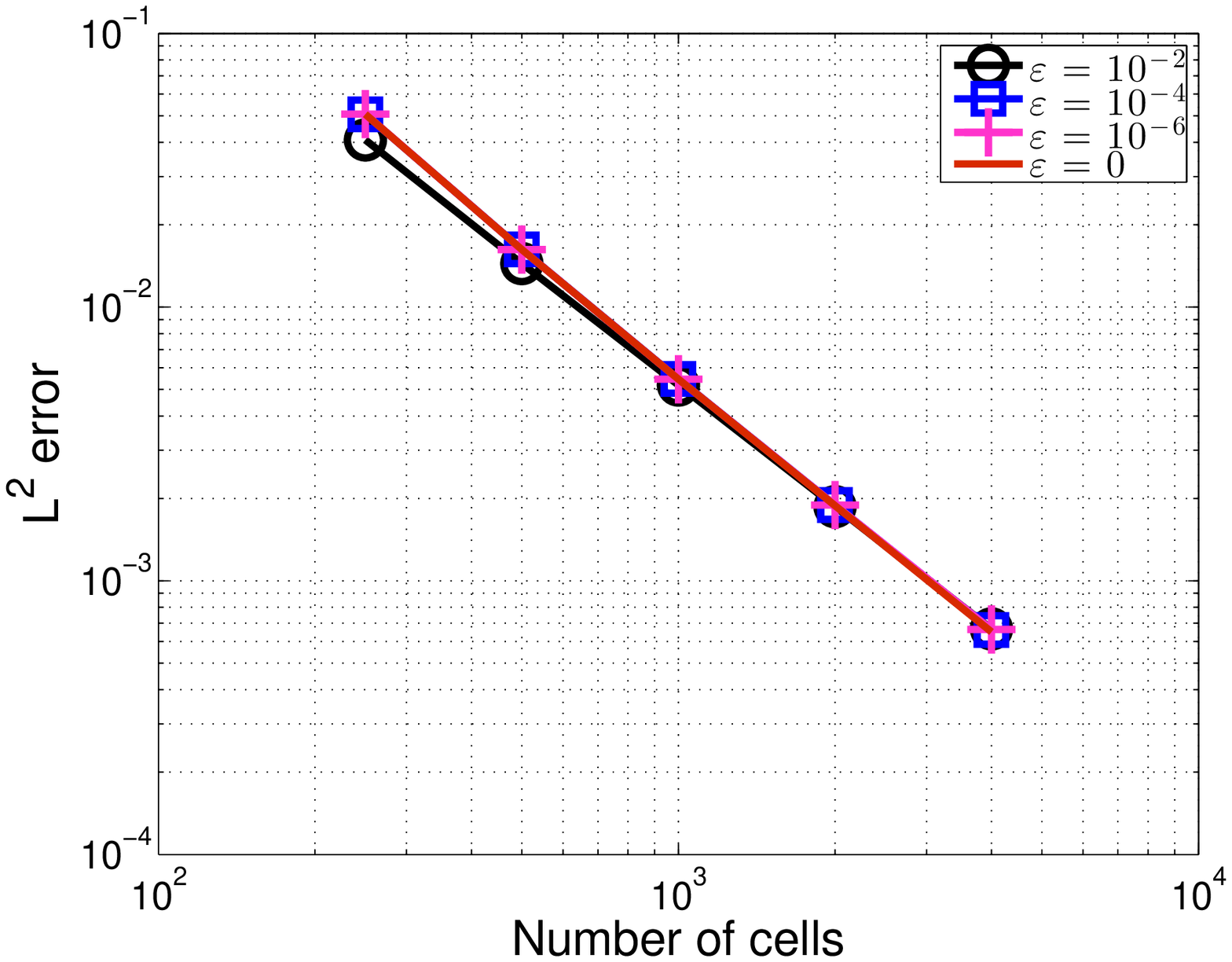}
\end{subfigure}
\begin{subfigure}[b]{0.45\textwidth}
\caption{$ \mathrm{L}^2 $ error on the electrons density}
\includegraphics[width=\textwidth]{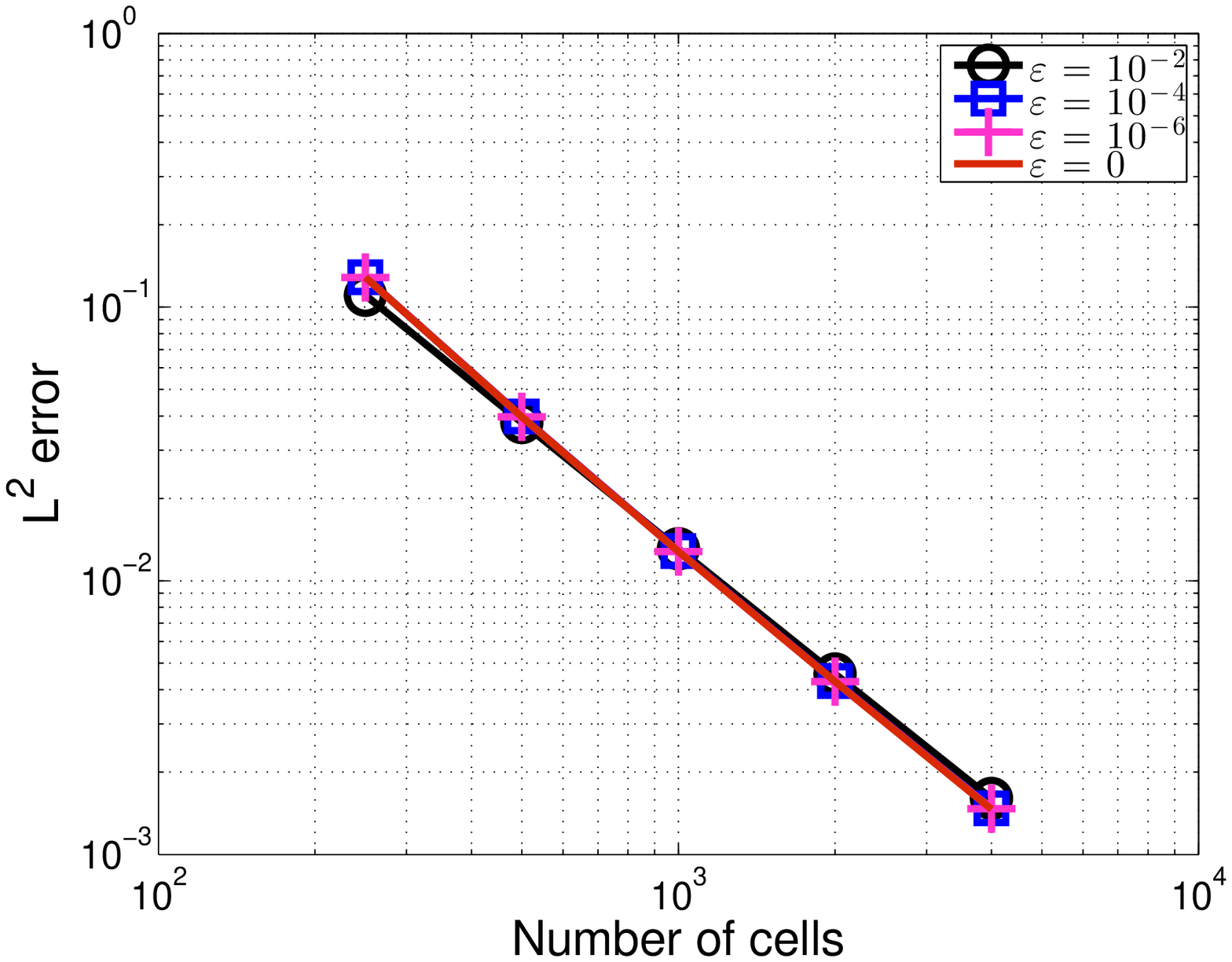}
\end{subfigure}
L
\begin{subfigure}[b]{0.45\textwidth}
\caption{$ \mathrm{L}^2 $ error on the electric potential}
\includegraphics[width=\textwidth]{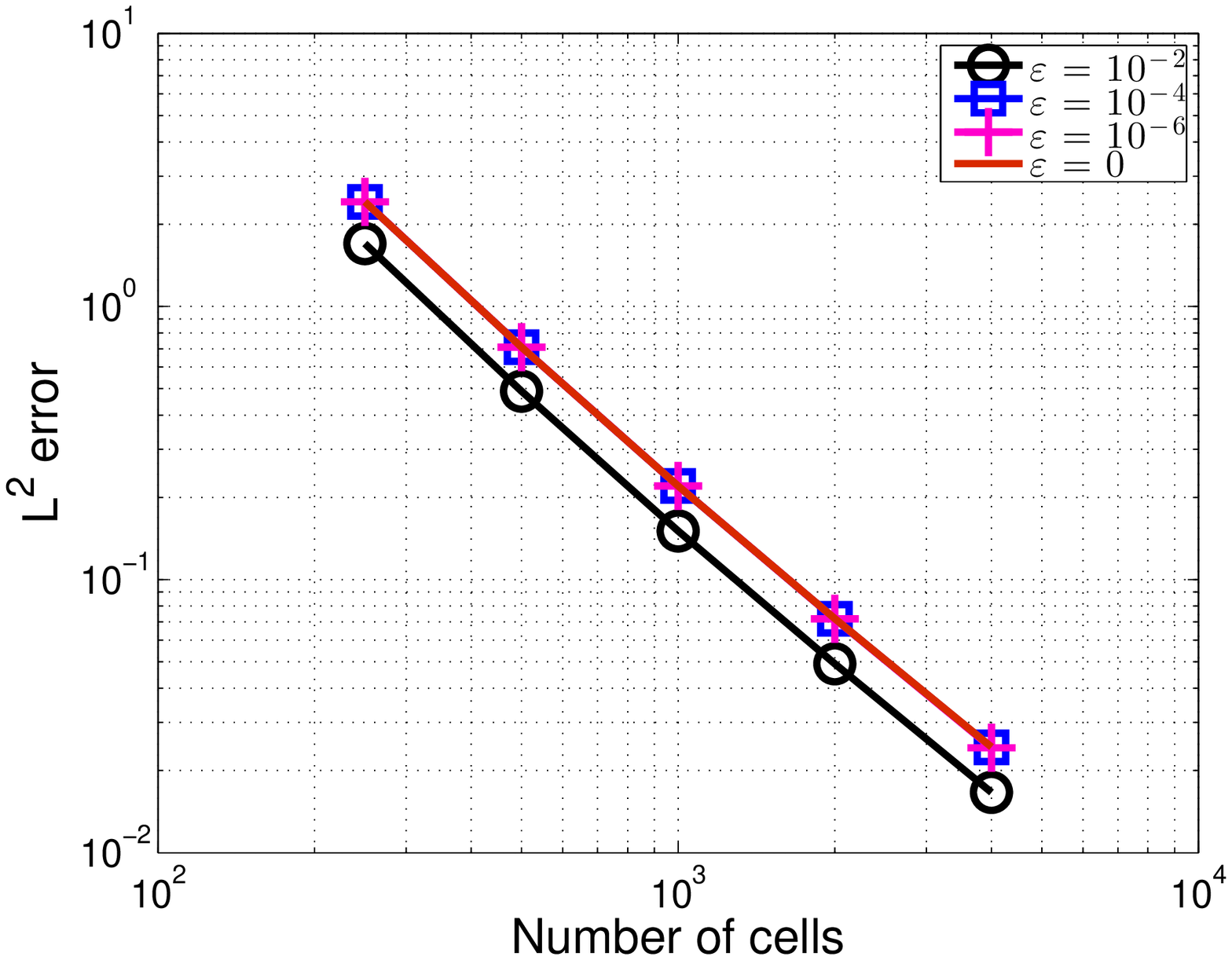}
\end{subfigure}
\end{center}
\caption{$ \mathrm{L}^2 $  errors with respect to the space step on the densities and the electric potential for different values of $ \varepsilon $ and at final time $ T= 10^{-3} $.} \label{erreurl2}
\end{figure}

In Figure \ref{densitepotentiel}, we present  the profiles of the two densities and the electric potential, in the case $ \varepsilon = 0 $.  In order to reach the stationary state, we use as final time $ T=1 $. The solution is computed on a uniform mesh with $ 2000 $ cells in space and with a time step $ \Delta t = 10^{-3} $. We obtain the expected behavior for the densities and the electric potential.

\begin{figure}[h!]
\begin{center}
\begin{subfigure}[b]{0.45\textwidth}
\caption{$ \mathrm{L}^2 $ error on the electrons density}
\includegraphics[width=\textwidth]{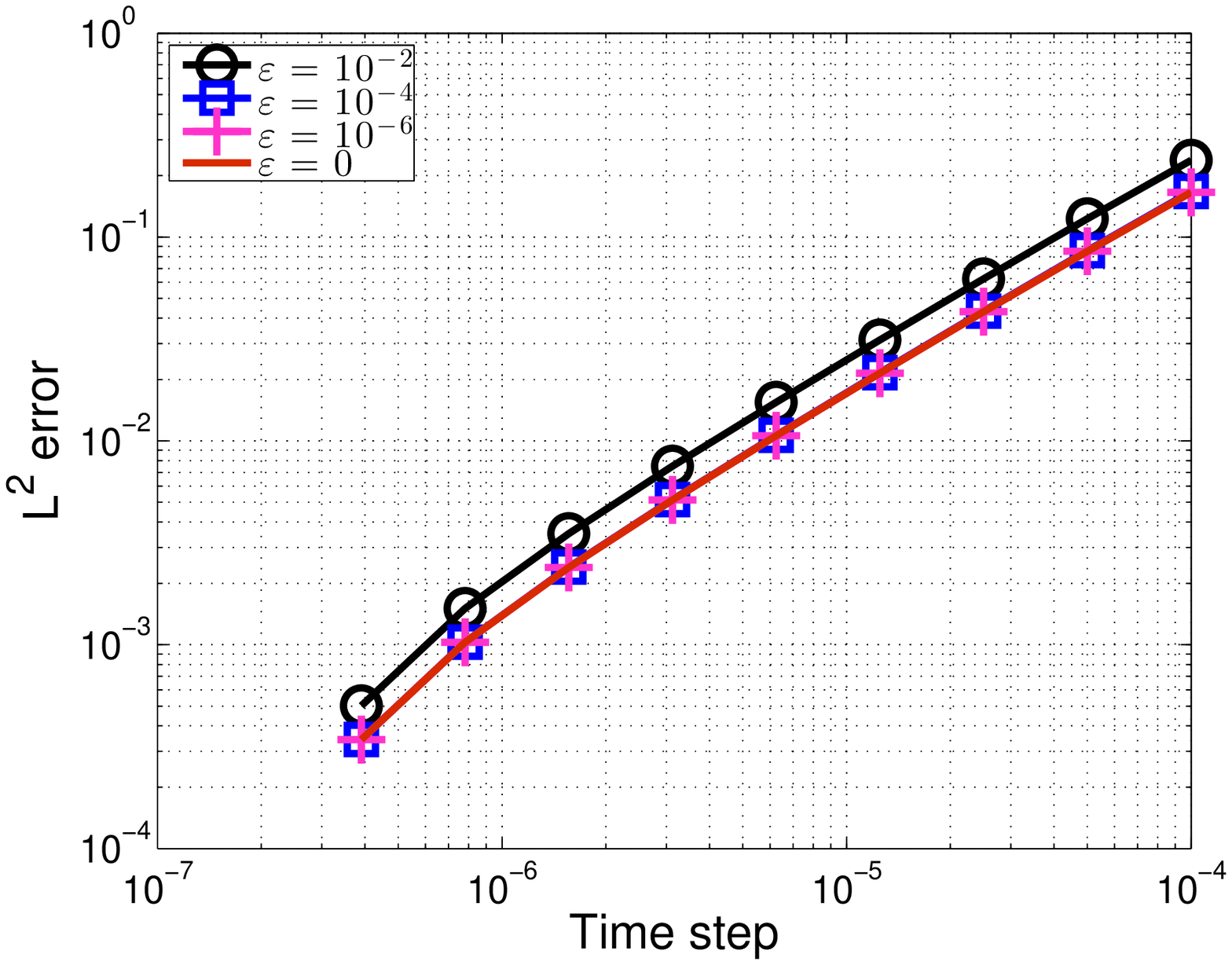}
\end{subfigure}
\begin{subfigure}[b]{0.45\textwidth}
\caption{$ \mathrm{L}^2 $ error on the cations density}
\includegraphics[width=\textwidth]{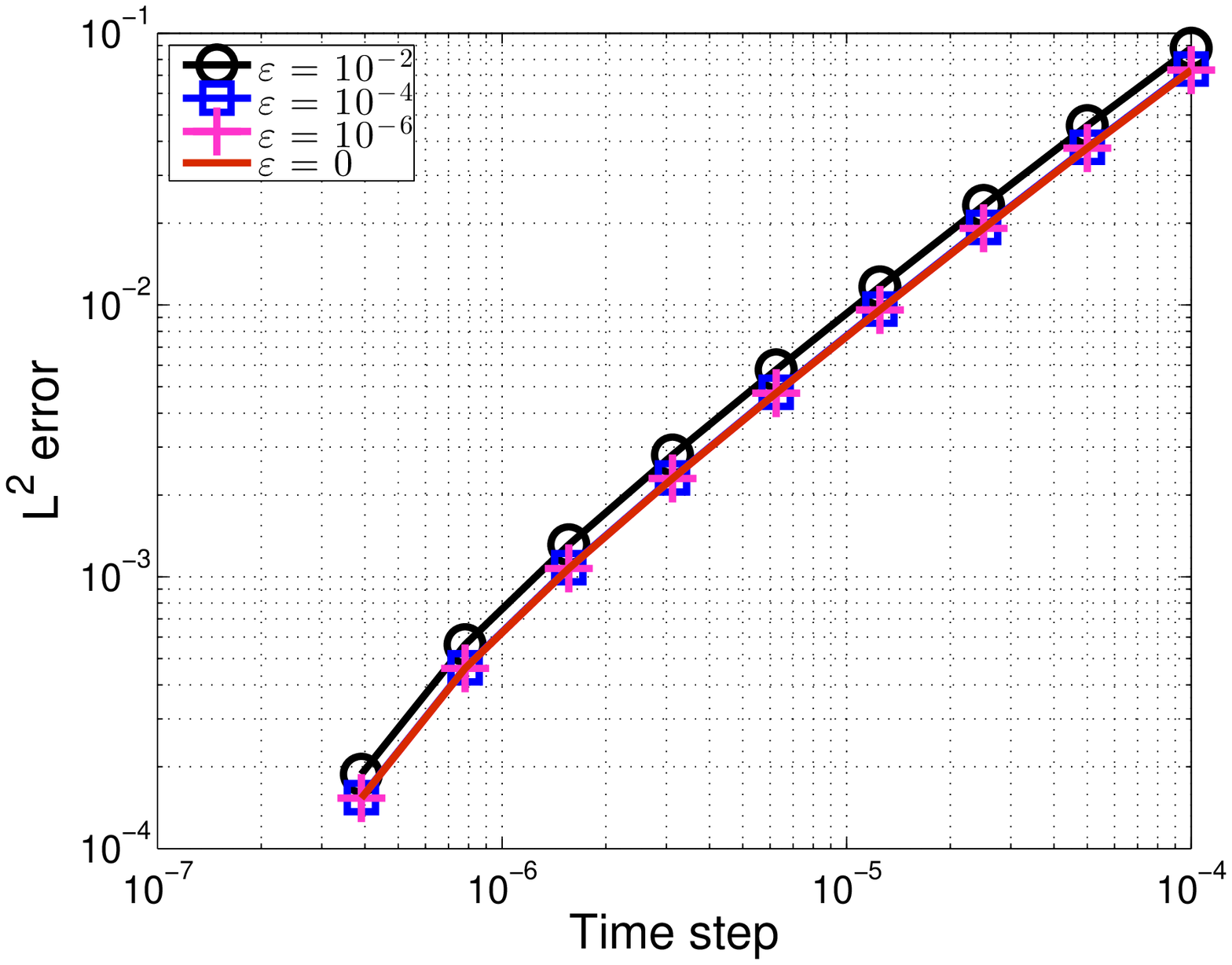}
\end{subfigure}
\begin{subfigure}[b]{0.45\textwidth}
\caption{$ \mathrm{L}^2 $ error on the electric potential}
\includegraphics[width=\textwidth]{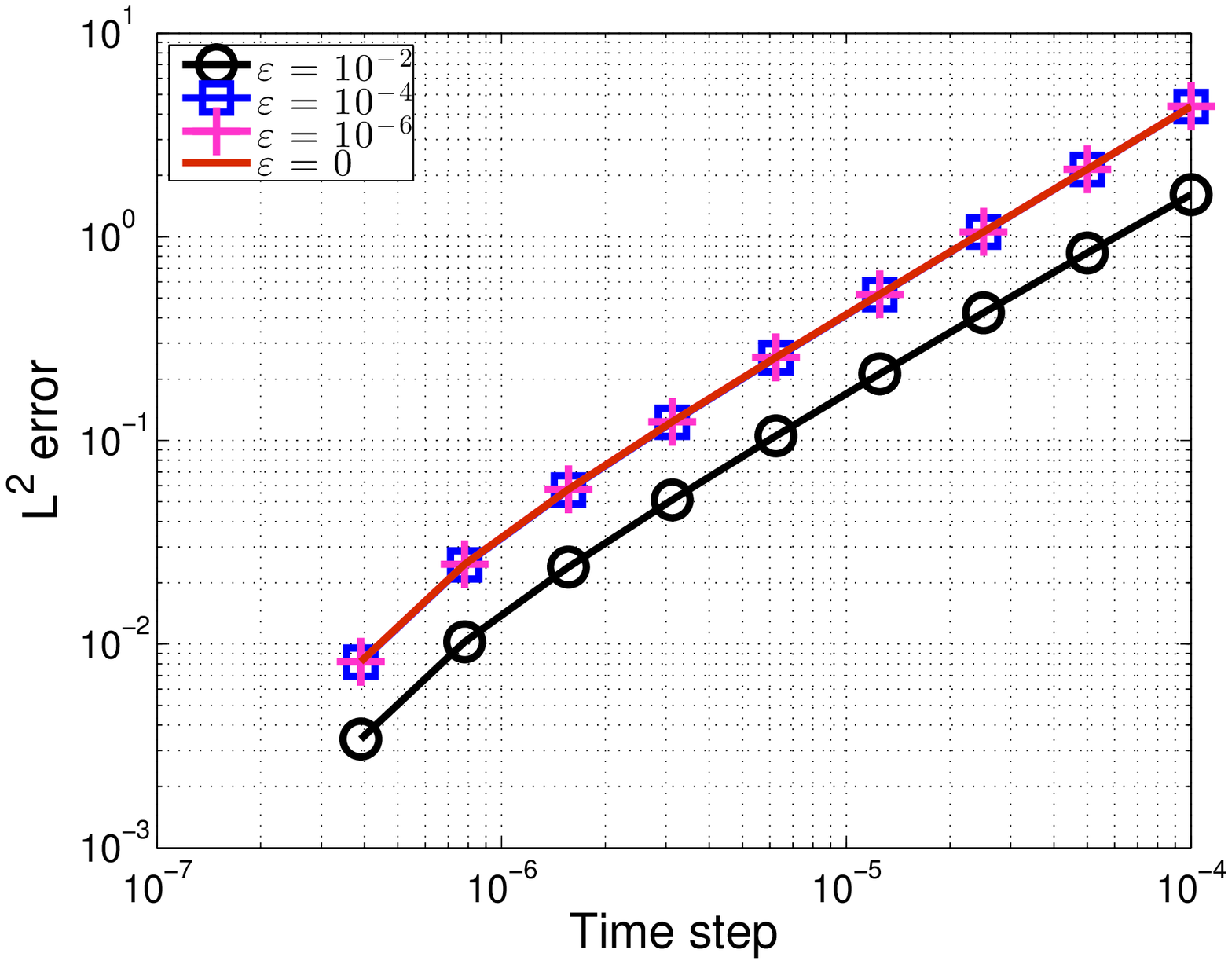}
\end{subfigure}
\end{center}
\caption{$ \mathrm{L}^2 $  errors with respect to the time step on the densities and the electric potential for different values of $ \varepsilon $ and at final time $ T= 10^{-3} $.} \label{erreurl2tps}
\end{figure}
\begin{figure}[h!]
\begin{center}
\begin{subfigure}[b]{0.45\textwidth}
\caption{$ \mathrm{L}^2 $ error on the electrons density}
\includegraphics[width=\textwidth]{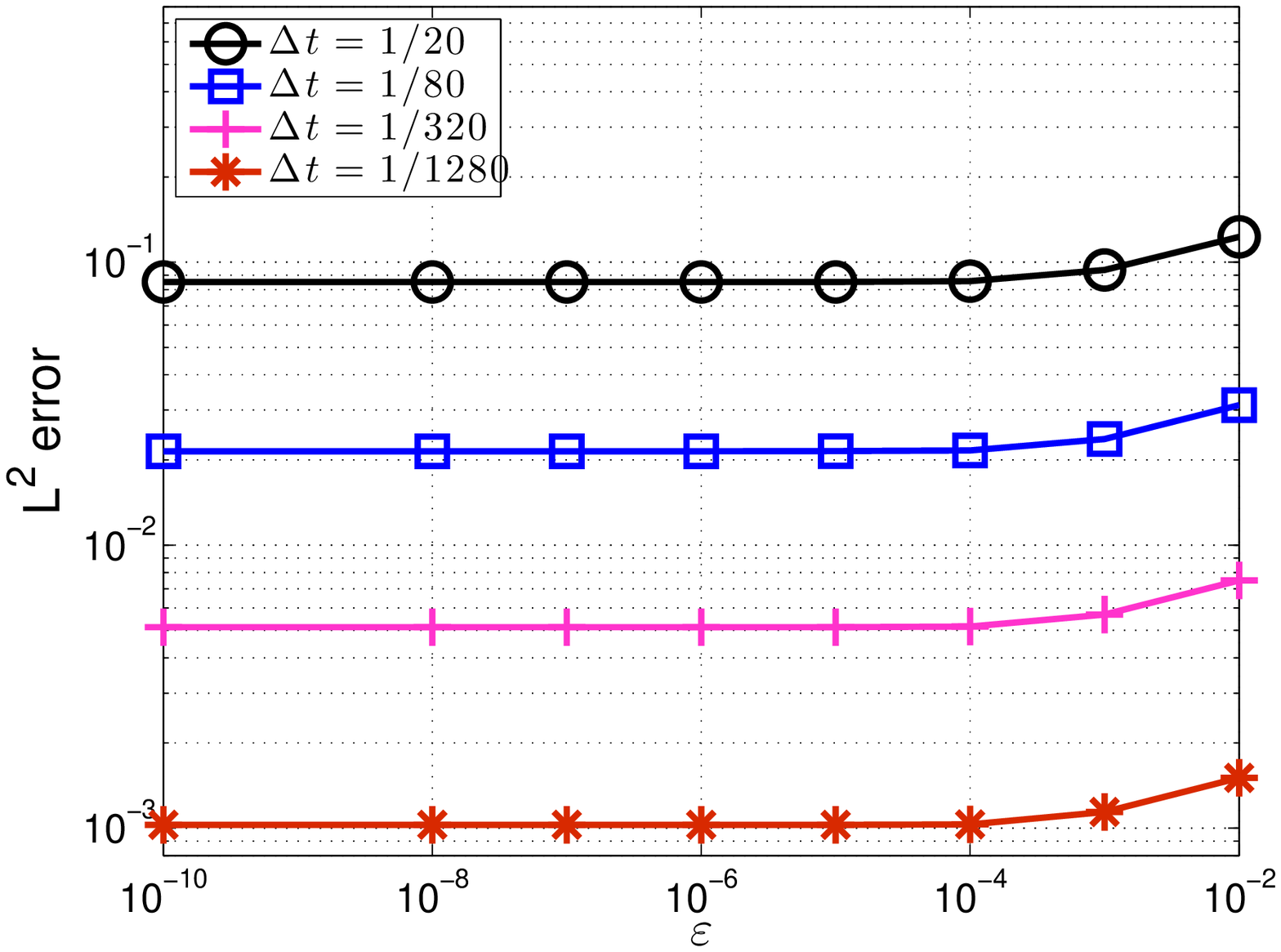}
\end{subfigure}
\begin{subfigure}[b]{0.45\textwidth}
\caption{$ \mathrm{L}^2 $ error on the cations density}
\includegraphics[width=\textwidth]{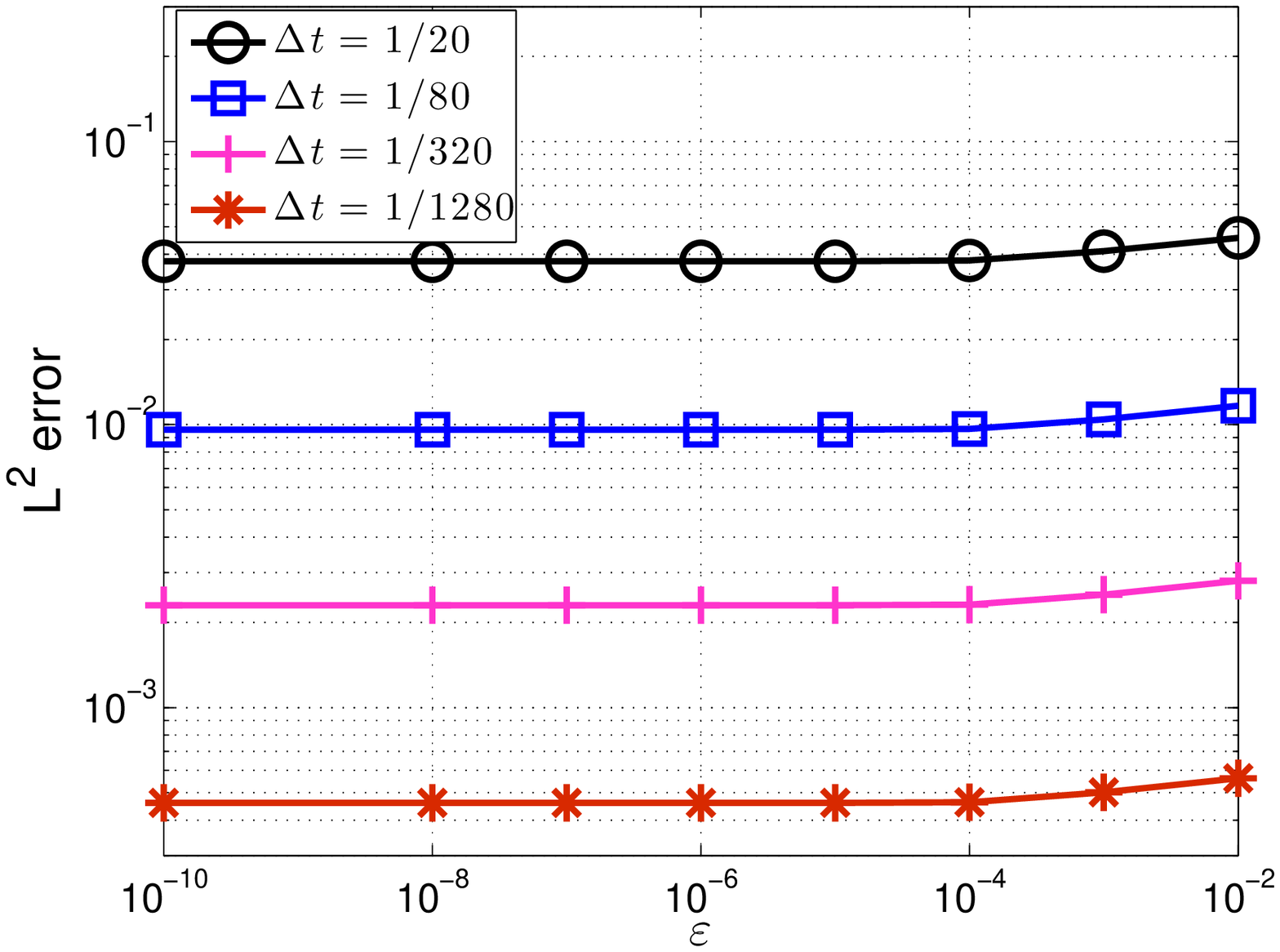}
\end{subfigure}
\begin{subfigure}[b]{0.45\textwidth}
\caption{$ \mathrm{L}^2 $ error on the electric potential}
\includegraphics[width=\textwidth]{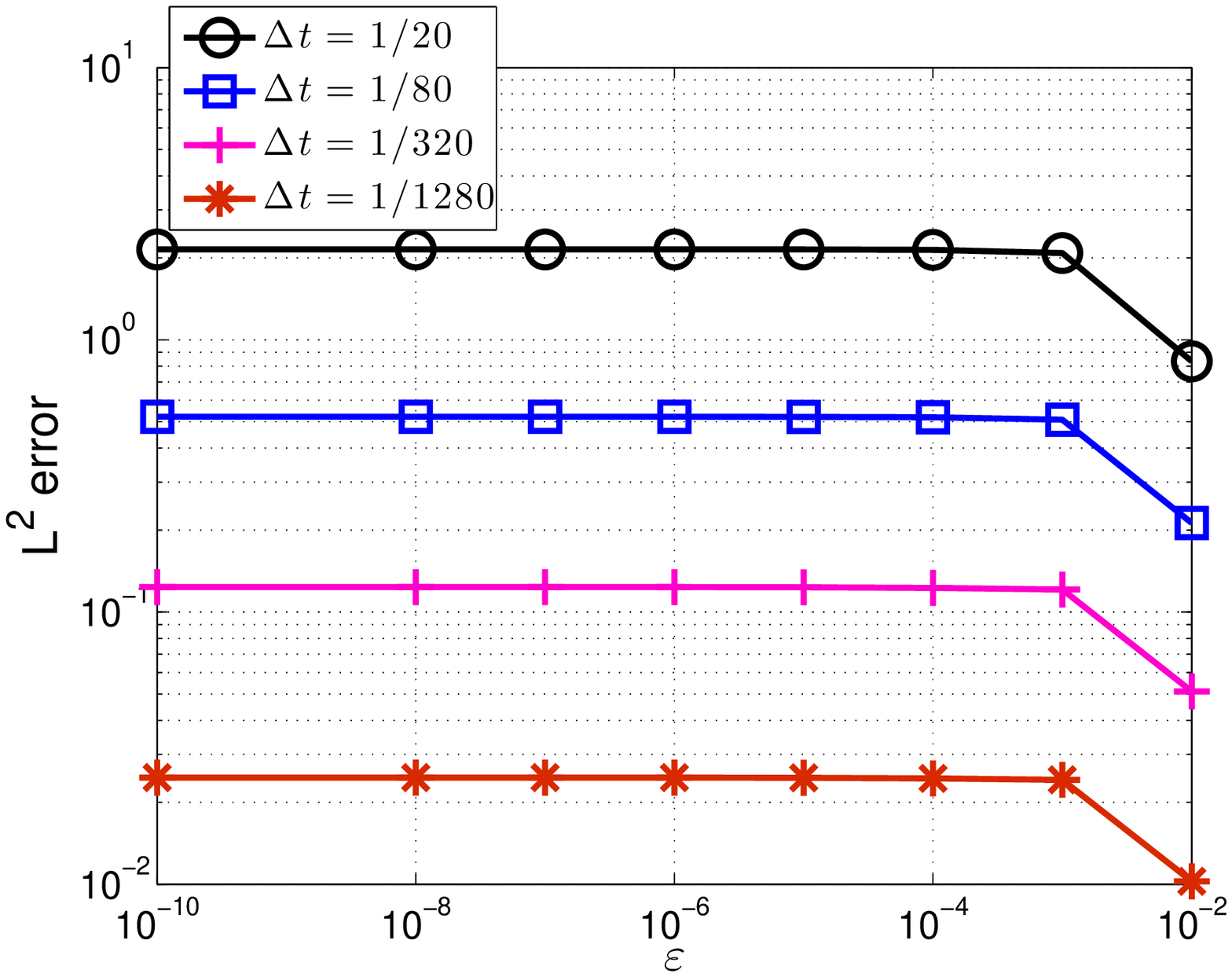}
\end{subfigure}
\end{center}
\caption{$ \mathrm{L}^2 $  error in space with respect to $ \varepsilon $ on the densities and the electric potential for different values of time step and at final time $ T= 10^{-3} $.} \label{erreurl2tpsepsi}
\end{figure}

%

\begin{figure}[h!]
\begin{center}
\begin{subfigure}[b]{0.45\textwidth}
\caption{Electrons density $ N $}
\includegraphics[width=\textwidth]{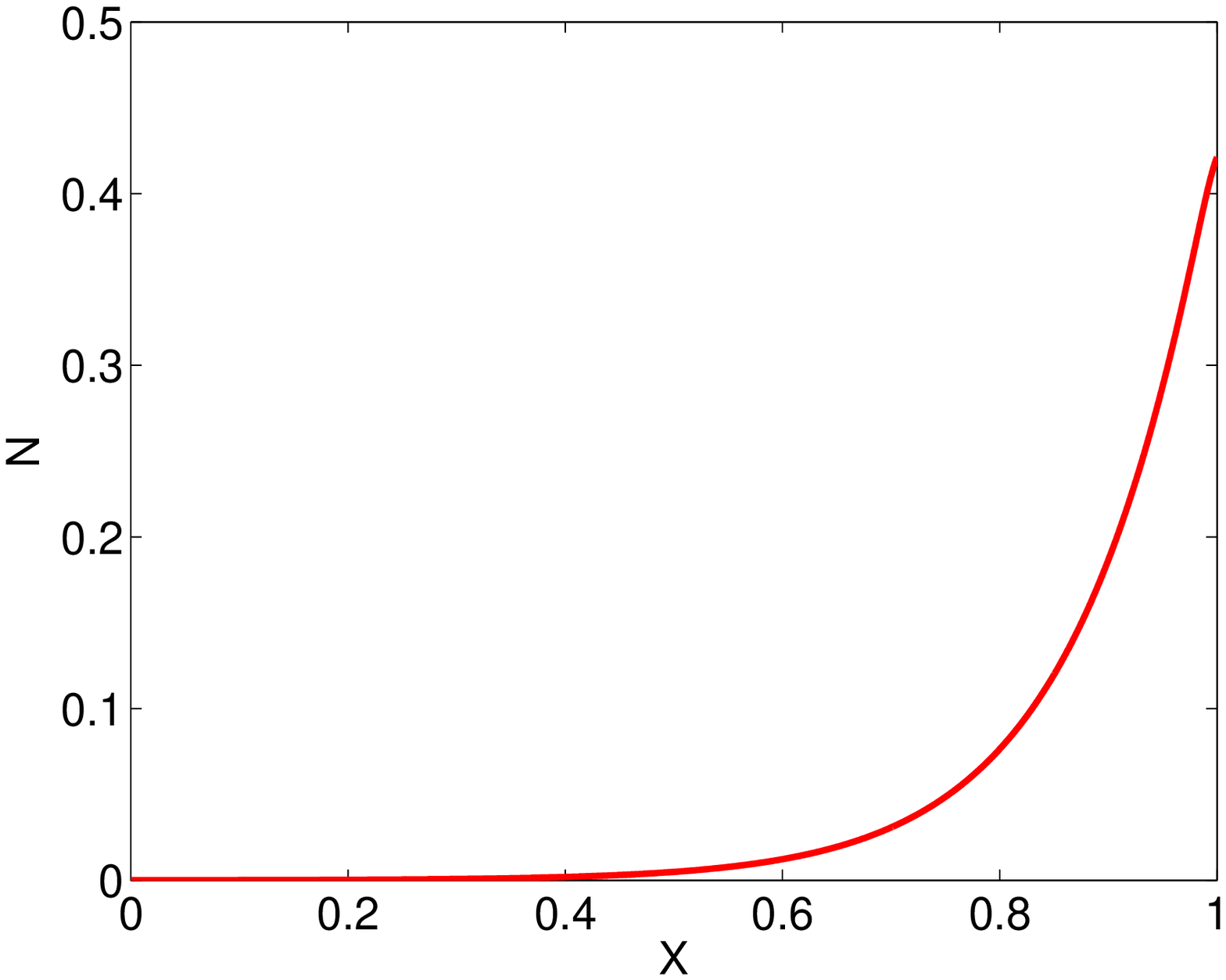}
\end{subfigure}
\begin{subfigure}[b]{0.45\textwidth}
\caption{Cations density $ P $}
\includegraphics[width=\textwidth]{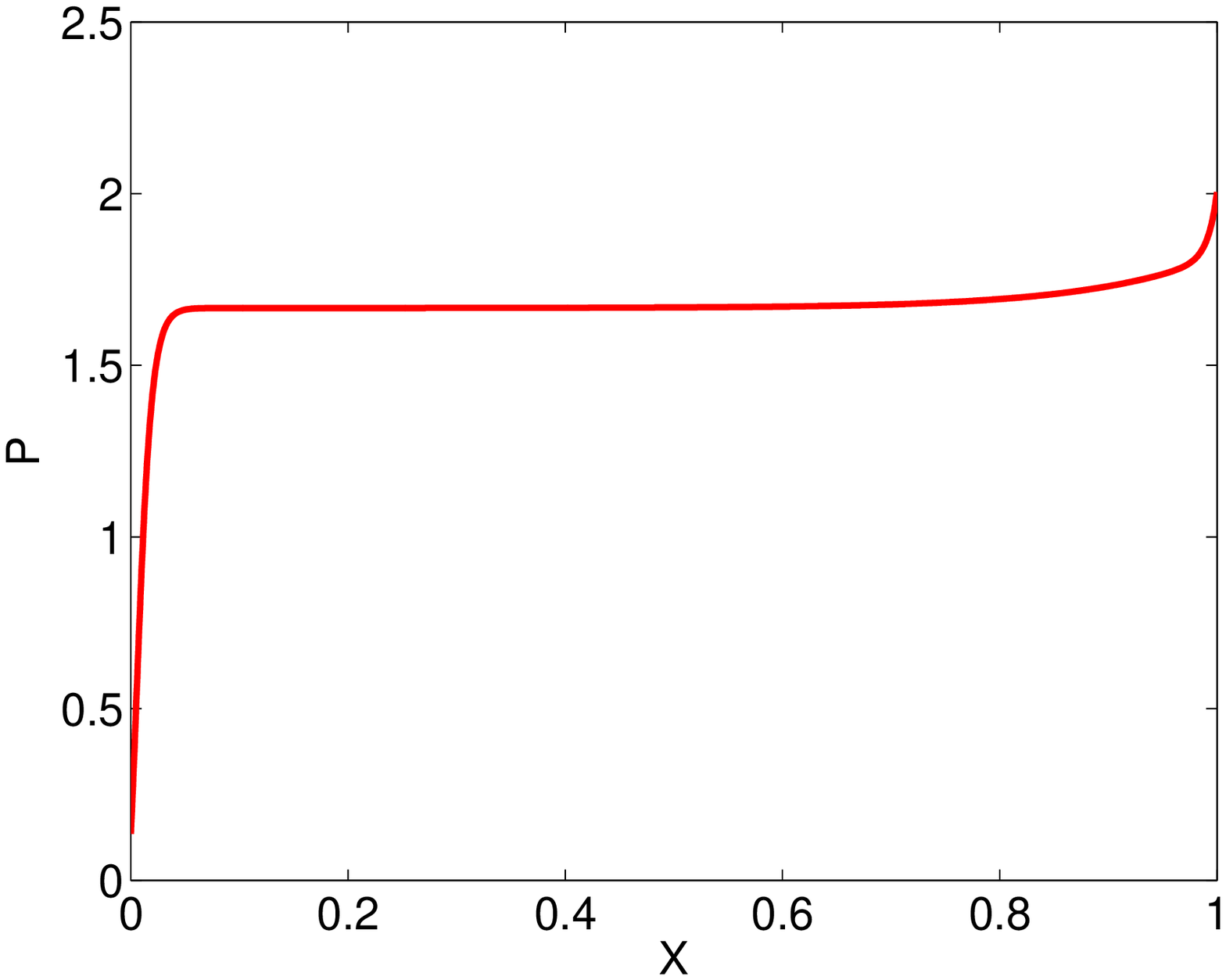}
\end{subfigure}
\begin{subfigure}[b]{0.45\textwidth}
\caption{Electric potential $ \Psi $}
\includegraphics[width=\textwidth]{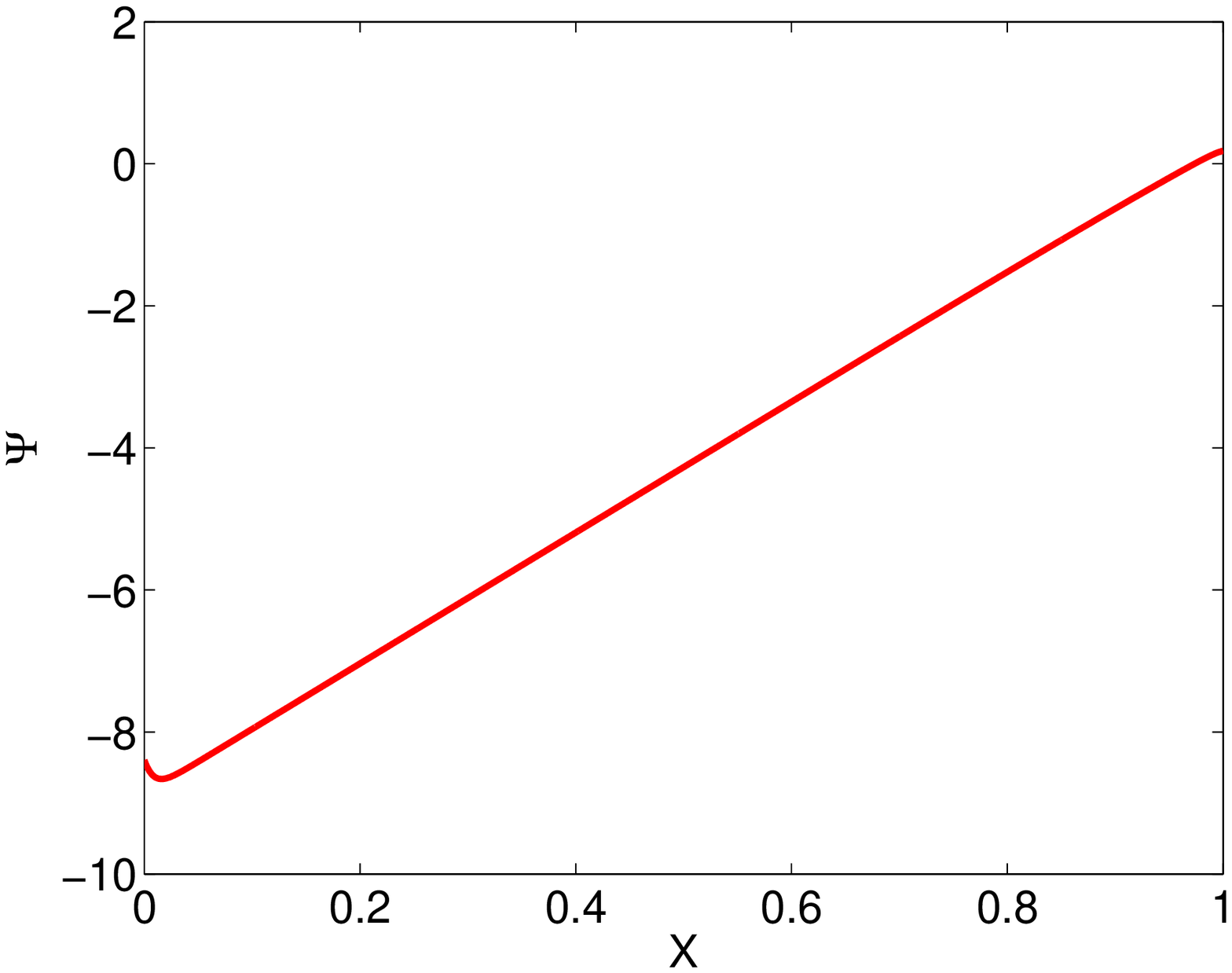}
\end{subfigure}
\end{center}
\caption{Densities and electric potential profiles for $ \varepsilon = 0 $, at final time $ T=1 $.} \label{densitepotentiel}
\end{figure}

\section{Conclusion}
In this paper, we prove the convergence of a numerical scheme consisting in an implicit Euler scheme in time and a Scharfetter-Gummel finite volume scheme in space for a corrosion model. The convergence proof is only valuable for $\varepsilon >0$. But, as mentioned in the introduction, the dimensionless parameter $\varepsilon$ is the ratio of the mobility coefficients for cations and electrons and  it is therefore very small. It should be set to 0 in the practical applications. 
 As seen in the numerical experiments, the scheme is robust with respect to the value of $\varepsilon$. It seems to be asymptotic preserving in the limit $\varepsilon\to 0$ (Figures \ref{erreurl2},\ref{erreurl2tps} and \ref{erreurl2tpsepsi}). In practice, we can set $\varepsilon=0$ and obtain the expected  behavior for the densities and the potential (Figure \ref{densitepotentiel}).

%


\section*{Acknowledgements}
This work was supported in part by the team INRIA/MEPHYSTO and the Labex CEMPI  (ANR-11-LABX-0007-01). 


\renewcommand{\refname}{7 \hskip0.5cm Bibliography}
\bibliographystyle{plain}
\bibliography{Biblio}

\end{document}